\documentclass{article}
\usepackage[utf8]{inputenc}
\usepackage{amsmath,amsfonts,amssymb}
\usepackage{mathrsfs}
\usepackage{mathabx}
\usepackage{wasysym}
\usepackage{amsthm,thmtools}
\usepackage{graphicx}
\usepackage{color}
\usepackage{verbatim}
\usepackage{tikz-cd}
\usepackage{tikz}
\usetikzlibrary{calc,tikzmark,shapes,arrows,decorations.pathmorphing}
\usepackage{xfrac}
\usepackage{bbm}
\usepackage[inline,shortlabels]{enumitem}
\usepackage{nccmath}		
\usepackage{xfrac}          
\usepackage[pdftex=true,hyperindex=true,pdfborder={0 0 0},colorlinks=true,allcolors=blue]{hyperref}
\usepackage{cleveref}
\usepackage[a4paper, left=2.8cm, right=2.8cm, top=3cm, bottom=2.5cm]{geometry}
\usepackage{mathtools}
\usepackage{figchild}
\usepackage{doi}

\theoremstyle{plain}
\newtheorem{theorem}{Theorem}[section] 
\newtheorem{knowntheorem}{Theorem}

\newtheorem{proposition}[theorem]{Proposition}
\newtheorem{lemma}[theorem]{Lemma}

\newtheorem{observation}[theorem]{Observation}

\theoremstyle{definition}
\newtheorem{definition}[theorem]{Definition} 
\newtheorem{notation}[theorem]{Notation}
\newtheorem{question}[theorem]{Question}

\declaretheorem[name=Example, sibling=theorem, qed=$\ocircle$]{example} 
\declaretheorem[name=Remark, sibling=theorem, qed=$\Diamond$]{remark} 

\newcommand{\Emph}[1]{\textbf{#1}}              

\newcommand{\QQ}{\mathbb{Q}}
\newcommand{\NN}{\mathbb{N}}
\newcommand{\ZZ}{\mathbb{Z}}
\newcommand{\GG}{\mathbb{G}}							
\newcommand{\HH}{{\mathbb{G}'}}							

\newcommand{\isdef}{\coloneqq}						
\DeclarePairedDelimiter\abs{\lvert}{\rvert}		
\DeclarePairedDelimiter\generatedby{\langle}{\rangle}		

\newcommand{\dd}{\mathrm{d}}			

\newcommand{\xv}[1]{\mathbf{#1}}                    
\newcommand{\indicator}[1]{\mathbbm{1}_{#1}}		
\newcommand\given[1][]{\,{#1\vert}\,}			
\newcommand{\field}[1]{\mathscr{#1}}                
\newcommand{\family}[1]{\mathscr{#1}}               

\DeclareMathOperator{\supp}{supp}		

\newcommand{\symb}[1]{\mathtt{#1}}		

\newcommand{\interior}[1]{\ring{#1}}        
\DeclareMathOperator{\Interior}{int}		

\newcommand{\critical}{\mathsf{c}}                  

\newcommand{\neighbours}{\mathcal{N}}               

\newcommand{\cayley}{\operatorname{Cay}}            

\newcommand{\pine}{\mathsf{p}}                      
\newcommand{\half}{{\sfrac{1}{2}}}

\newcommand{\xint}{\mathsf{in}}                 
\newcommand{\xext}{\mathsf{ex}}                 

\newcounter{sarrow}

\title{Cellular automata, percolation and dynamical dichotomies}
\author{Sebasti\'an Barbieri, Felipe Garc\'ia-Ramos and Siamak Taati}
\date{}

\newcommand{\Addresses}{{
    \footnotesize
    \begin{samepage}
        \noindent S.~Barbieri\\
        \textsc{Departamento de Matem\'{a}tica y ciencia de la computaci\'{o}n, Universidad de Santiago de Chile, Santiago, Chile.}\\
        \indent\emph{E-mail address}: \texttt{\href{mailto:sebastian.barbieri@usach.cl}{sebastian.barbieri@usach.cl}}
    \end{samepage}
		
	\medskip

    \begin{samepage}
        \noindent F.~Garc\'ia-Ramos\\
        \textsc{Physics Institute, Universidad Autónoma de San Luis Potosí, San Luis Potosí, Mexico.}\\
        \textsc{Faculty of Mathematics and Computer Science, Jagiellonian University,  Kraków, Poland.}\\
        \indent\emph{E-mail address}: \texttt{\href{mailto:fgramos@conahcyt.mx}{fgramos@conahcyt.mx}}
    \end{samepage}
    
    \medskip

    \begin{samepage}
        \noindent S.~Taati\\
        \textsc{Department of Mathematics, American University of Beirut, Beirut, Lebanon.}\\
        \textsc{Center for Advanced Mathematical Sciences, American University of Beirut, Beirut, Lebanon.}\\
        \indent\emph{E-mail address}: \texttt{\href{mailto:siamak.taati@gmail.com}{siamak.taati@gmail.com}}
    \end{samepage}

    \medskip
}}

\begin{document}
	
	\maketitle
\begin{abstract}
    We establish a connection between percolation on the Cayley graphs of a group and the dynamical diversity of cellular automata on that group. Specifically, we demonstrate that Gilman's dichotomy between equicontinuity and sensitivity with respect to Bernoulli measures holds on a finitely generated group if and only if the group has a trivial percolation threshold. Consequently, we show that a countable group satisfies Gilman's dichotomy if and only if it is locally virtually cyclic.
    
		\medskip
		
		\noindent
		\emph{Keywords:} group theory, cellular automata, percolation, sensitivity,  equicontinuity, Bernoulli measures, random initial conditions.
		
		\smallskip
		
		\noindent
		\emph{MSC2020:} \textit{Primary:}
        37B15, 
        37A50. 
		\textit{Secondary:}
        60K35. 

        \vspace{-3em}
        \renewcommand{\contentsname}{}
        {\footnotesize\tableofcontents}

\end{abstract}
\section{Introduction}

Cellular automata (CA) are discrete-time, continuous, shift-equivariant dynamical systems on configurations of symbols on a lattice. 
The study of CA on groups has revealed deep connections between symbolic dynamics and the algebraic, recursive, and geometric properties of groups
(see~\cite{ceccherini2010cellular, CC2023} and the references therein).
In this paper, we identify a new connection between the dynamical properties of CA on groups and percolation theory.


A central theme in the study of complex systems is the interplay between order and chaos (or between structure and randomness).
In the context of dynamical systems, these are manifested in the concepts of sensitivity (to initial conditions) and equicontinuity.
Sensitivity refers to the scenario in which small perturbations can lead to significant changes in a system's evolution. The study of sensitivity goes back to the works of Lorenz and it is considered to be one of the defining characteristics of chaos~\cite[Section~1.8]{Devaney1989}.
Equicontinuity describes the stability of the trajectory of a system in response to small perturbations of the system's initial condition.
It generalizes the notion of Lyapunov stability to arbitrary trajectories of a topological dynamical system.


It has long been observed that, in the context of CA, many dynamical properties have natural interpretations in terms of the propagation of information in the configurations of the system. For instance, sensitivity can be characterized in terms of whether local information about the initial condition can propagate arbitrarily far throughout the system, while equicontinuity means the information in the initial condition of the CA remains localized.


To understand the typical behaviour of a CA, it is natural to study its evolution with random initial conditions. 
Random initial conditions are typically modeled by a Bernoulli measure on the configuration space (i.e., a probability measure induced by an i.i.d.\ process).
In this context, Gilman introduced notions of sensitivity and equicontinuity with respect to a Bernoulli measure~$\mu$ ($\mu$-sensitivity and $\mu$-equicontinuity for short).
In contrast to the purely topological notions which consider all perturbations, Gilman's notions are concerned with sensitivity and stability with respect to ``typical'' perturbations.
In this context, Gilman proved that, in the one-dimensional case (i.e., when the underlying group is~$\ZZ$), every CA falls into one (and only one) of these two categories~\cite{Gilman1987}.

\begin{knowntheorem}[Gilman's dichotomy]
\label{thm:gilman}
    Let $\varphi\colon A^\ZZ\to A^\ZZ$ be a CA and $\mu$ a Bernoulli measure on~$A^\ZZ$.
    Then, $\varphi$ is either $\mu$-equicontinuous or $\mu$-sensitive.
\end{knowntheorem}

\noindent 

The precise definitions of $\mu$-sensitivity and $\mu$-equicontinuity can be found in Section~\ref{sec:sensitivityandequicontinuity} (Definitions~\ref{def:musens} and~\ref{def:mueq}).
Informally, $\mu$-sensitivity means there exists a finite region $F$ in the group $\GG$ such that, as the CA evolves, two independent random initial conditions will eventually exhibit distinct values within $F$, that is, their distinction almost surely propagates to~$F$. 
Conversely, $\mu$-equicontinuity implies that, given a finite region $F$, two random configurations that agree on a large area around $F$ will have the same values on $F$ throughout the entire evolution of the CA with high probability. In other words, information remains localized with high probability. 


Let us emphasize that in our current setting, the measure $\mu$ is not required to be invariant under the CA.
In fact, according to a general result of Huang, Lu and Ye~\cite{HLY2011}, a similar dichotomy between $\mu$-equicontinuity and $\mu$-sensitivity holds for all topological dynamical systems as long as $\mu$ is invariant and ergodic under the dynamics.

The geometry of the underlying group $\GG$ imposes restrictions on how information may propagate, and consequently, on the dynamics of the CA. For instance,
Shereshevsky showed that a CA on $\ZZ^d$ cannot be positively expansive unless $d=1$~\cite{Shereshevsky1993}. In this paper, we address the following question:
\begin{question}
    For which groups does Gilman's dichotomy hold?
\end{question}

We show that Gilman's result extends to CA on virtually $\ZZ$ groups and holds for any measure which is ergodic under the shift action 
on the configuration space
(Theorem~\ref{thm:virtually_Z_satisfies_dichotomy}). The main result of this article is that the validity of Gilman's dichotomy on a finitely generated group $\GG$ is tied to the non-triviality of the percolation threshold on the Cayley graphs of~$\GG$.



Percolation theory is the study of connectivity in the random subgraphs of an infinite locally finite graph obtained by deleting vertices (or edges) independently at random.
The central question in this field is whether or not such a random graph will contain an infinite connected component.
If such an infinite connected component exists, we say that the random graph \emph{percolates}.
Percolation models are among the simplest models in probability theory and statistical physics that exhibit phase transitions.
Let $p\in[0,1]$ denote the probability that each vertex is kept.
It can be shown that, as the parameter $p$ is varied from $0$ to~$1$, a transition occurs from the almost sure absence to the almost sure presence of percolation.
The critical value that witnesses this transition is referred to as the \emph{percolation threshold}.
It is easy to see that the percolation threshold of the Cayley graph of~$\ZZ$ with any finite generating set is~$1$.  In contrast, every lattice of dimension larger than or equal to~$2$ has a \emph{non-trivial} percolation threshold.
Recently, Duminil-Copin, Goswami, Raoufi, Severo, and Yadin proved that a Cayley graph of a finitely generated group $\GG$ has a trivial percolation threshold if and only if the group $\GG$ is virtually cyclic~\cite{DGR+2020} (Theorem~\ref{thm:percolation} below). 


We now introduce our general counter-example to Gilman's dichotomy.
This example can be thought of as an additive CA on a percolated environment.



\begin{example}[Percolated additive CA]
\label{exp:random-percolated additive-CA}
    Let $S$ be a finite generating set for a group~$\GG$, and let $A\isdef\{\symb{0},\symb{1}\}\times\{\symb{0},\symb{1}\}^S$.
    Each configuration in $A^\GG$ can be viewed as a pair of configurations $(x,w)$ where $x\in\{\symb{0},\symb{1}\}^\GG$ and $w\in\{\symb{0},\symb{1}\}^{\GG\times S}$.
    The \emph{percolated additive} CA is the CA defined by the map $\varphi\colon A^\GG\to A^\GG$, where
    \begin{align*}
        \varphi(x,w)_g &\isdef
            \bigg(\Big(\sum_{s\in S} w_{g}(s)\cdot x_{gs}\Big) \bmod{2}, w_g\bigg),
            \qquad\text{for every $g \in \GG$.}
    \end{align*}
    Here, we view the directed edge $(g,gs)$ as \emph{open} whenever $w_g(s)= \symb{1}$ and \emph{closed} otherwise. Thus, the open edges remain open, and the closed edges remain closed as the CA evolves. The state of each site (i.e., the value of its first component) is updated to the sum modulo $2$ of its neighbors in the subgraph defined by the open edges.
\end{example}



It is easy to show that the percolated additive CA cannot be sensitive with respect to any fully supported measure (Proposition~\ref{prop:random-percolated additive:not-mu-sensitive}). We will demonstrate that when $\GG$ has a non-trivial percolation threshold, one can choose the generating set $S$ in such a way that the percolated additive CA on $\GG$ is not equicontinuous with respect to the uniform Bernoulli measure.
Hence, Gilman's dichotomy fails on any such group. Using the previously mentioned characterization of the groups with non-trivial percolation threshold, we conclude the following.

\begin{theorem}
\label{thm:main-result}
    An infinite finitely generated group satisfies Gilman's dichotomy if and only if it is virtually $\ZZ$.
\end{theorem}

\noindent The proof of this theorem appears in Section~\ref{sec:Gilmandichotomy}.

In Section~\ref{sec:gilman-on-countable-groups}, we extend this characterization to cover all countable groups.
This is done via reduction to the finitely generated case.
We recall that a group $\GG$ is \emph{locally virtually cyclic} if every finitely generated subgroup of $\GG$ is either finite or contains an isomorphic copy of $\ZZ$ as a finite index subgroup.

\begin{theorem}
\label{thm:main2}
    A countable group satisfies Gilman's dichotomy if and only if it is locally virtually cyclic.
\end{theorem}



We remark that the topological (non-stochastic) concepts of sensitivity and equicontinuity have been thoroughly studied for CA on~$\ZZ$. In this context, K\r{u}rka has established a dichotomy similar to Gilman's dichotomy, proving that every CA on $\ZZ$ is either sensitive to initial conditions or has a residual set of equicontinuity points~\cite{Kurka1997}. In contrast, Sablik and Theyssier have provided an example of a CA on $\ZZ^2$ that is not sensitive and has no equicontinuity points~\cite{ST2010}.
It is conjectured that such CA exist on all countable groups that are not virtually cyclic~\cite[Conjecture 5.2.25]{Bitar-thesis_2024}.
This and other open questions are discussed in Section~\ref{sec:remarks}.

\paragraph{Acknowledgements.}
The authors would like to thank Nicanor Carrasco-Vargas for allowing us to use his proof of Lemma~\ref{lem:fat_nonamenable}, and Nicol\'as Bitar and Gourab Ray for providing helpful references.  We also thank the organizers of the thematic month ``Discrete Mathematics \& Computer Science: Groups, Dynamics, Complexity, Words'' at the Centre International de Rencontres Mathématiques (CIRM), Marseille, during which this project was partly developed.

S.~Barbieri was supported by the ANID grant FONDECYT regular 1240085, the DICYT grant 606, and the Institut de Mathématiques de Marseille, Aix-Marseille Université. 

F.~García-Ramos was supported by the Grant U1U/W16/NO/01.03 of the Strategic Excellence Initiative program of
the Jagiellonian University, and the grant K/NCN/000198 of the Narodowe Centrum Nauki. 

S.~Taati was supported by the Center for Advanced Mathematical Sciences (CAMS), American University of Beirut (CAMS ORCID: 0009-0004-5763-5004), and the Institut de Mathématiques de Marseille, Aix-Marseille Université.




\section{Preliminaries}
\label{sec:preliminaries}
 
We use the notation $A\Subset B$ to indicate that $A$ is a finite subset of $B$.

\subsection{Graphs and groups}

In this paper, a \Emph{graph} refers to a countable directed graph, that is, a pair $(V,E)$, where $V$ is a countable set representing the \Emph{vertices}, and $E\subseteq V\times V$, represents the \Emph{edges} of the graphs.
We will denote the set of the vertices of a graph $\Gamma$ by $V(\Gamma)$ and the set of its edges by $E(\Gamma)$.
We sometimes refer to vertices as \Emph{sites} and to edges as \Emph{bonds}.

A finite \Emph{path} in a graph~$\Gamma$ refers to a sequence $(v_1,v_2,\ldots,v_m)$ of vertices in $V(\Gamma)$ such that $(v_i,v_{i+1})\in E(\Gamma)$ for each $i$. An infinite path is defined analogously.
A path is said to be \Emph{self-avoiding} if it does not visit any vertex more than once.

Throughout this paper, $\GG$ stands for a countable group with identity element~$e$.
If $\GG$ is finitely generated and we consider a finite set of generators (or generating set) $S$, we will always implicitly assume that $S$ is \Emph{symmetric}, that is, closed by inverses. We denote the \Emph{Cayley graph} of $\GG$ associated to a finite generating set~$S$ by $\cayley(\GG,S)$. To recall, this is the graph with vertex set $\GG$ and edge set $\big\{(g,gs): \text{$g\in \GG$ and $s\in S$}\big\}$.
The \Emph{length} of an element $g\in\GG$ with respect to $S$, denoted by $\abs{g}_S$, refers to the length of the shortest representation of $g$ as a product of elements from~$S$.
This value coincides with the length of the shortest path from $e$ to $g$ in $\cayley(\GG,S)$.
The \Emph{ball} of radius $n$ around $g\in\GG$ in $\cayley(\GG,S)$ is the set $B_n(g)\isdef\{gh: \abs{h}_S\leq n\}$.
The \Emph{centered} ball of radius~$n$ is $B_n\isdef B_n(e)$. 

Given $g\in\GG$ and $B\subseteq\GG$, we use the notation $gB\isdef\{gh: h\in B\}$.  Similarly, given $A,B\subseteq\GG$, we define $AB\isdef\{gh: \text{$g\in A$ and $h\in B$}\}$.

Recall that $g\in\GG$ is a \Emph{torsion element} if there exists an integer $n\geq 1$ such that $g^n=e$. A group $\GG$ is called \Emph{virtually $\ZZ$} if it contains a finite index subgroup which is isomorphic to $\ZZ$. Equivalently, $\GG$ is virtually $\ZZ$ if it contains a non-torsion element $h\in \GG$ such that $\generatedby{h} \isdef \{ h^{k} : k \in \ZZ\}  $ is a finite index subgroup of $\GG$.

A group $\GG$ is called \Emph{virtually cyclic} if it contains a cyclic subgroup of finite index, that is, if $\GG$ is either finite or virtually $\ZZ$. A \Emph{locally virtually cyclic} group is a group whose finitely generated subgroups are all virtually cyclic.

For further background on the geometric aspects of countable groups, we refer to the monograph by Meier~\cite{Meier2008}.

 \subsection{Cellular automata}
 \label{sec:prelim:ca}

Let $G$ be a countable set, and $A$ be a finite set with $\abs{A}\geq 2$, which we call an \Emph{alphabet}.  A map $x\colon G\to A$ is referred to as a \Emph{configuration} (on $G$).  For a configuration $x\in A^{G}$ and an element  $g\in G$, we use the notations $x_g$ and $x(g)$ interchangeably, and interpret it as the symbol at position~$g$.
The restriction of a configuration $x\in A^{G}$ to a set $F\subseteq G$ will be denoted by $x_F$.
A map $w\colon F\to A$ with $F\Subset G$ is called a (finite) \Emph{pattern}.
Every pattern $w\colon F\to A$ defines a set
\begin{align*}
    [w] &\isdef \{x\in A^{G}: x_F=w\} \;,
\end{align*}
which is called a \Emph{cylinder}.

The set $A^{G}$ of all configurations is endowed with the product topology inherited from the discrete topology on~$A$.
We also equip $A^{G}$ with the Borel $\sigma$-algebra which, in the current setting, coincides with the product $\sigma$-algebra when $A$ is given the discrete $\sigma$-algebra.
The cylinders form a basis for the topology on $A^{G}$.  Furthermore, every probability measure on~$A^{G}$ is uniquely determined by the probabilities it assigns to the cylinders.

We are mostly interested in the case when $G$ is a group, in which case there is a natural left action by translations.
As before, let $\GG$ be a countable group.
The \Emph{shift} action of $\GG$ on $A^\GG$ is given by the map $(g,x)\mapsto gx$, where $(gx)_h\isdef x_{g^{-1}h}$.
A map $\varphi \colon A^{\GG}\rightarrow A^{\GG}$ is $\GG$-\Emph{equivariant} if $\varphi(gx)=g\varphi(x)$ for all $g\in \GG$ and $x\in A^{\GG}$. 

\begin{definition}
\label{def:CA}
    A \Emph{cellular automaton} (\Emph{CA}) is a map $\varphi \colon A^{\GG}\to A^{\GG}$ that is continuous and $\GG$-equivariant. 
\end{definition}
\noindent We are interested in the dynamical system obtained by iterating~$\varphi$ on an initial configuration.

The following theorem, due to Curtis, Hedlund and Lyndon~\cite{Hedlund1969}, provides an equivalent definition for CA.
\begin{knowntheorem}[Local description of CA]
\label{thm:curtis}
    A map $\varphi \colon A^{\GG}\rightarrow A^{\GG}$ is a CA if and only if there exists a set $K\Subset G$ and a function $f\colon A^K\to A$ such that 
    \begin{align*}
        \varphi(x)_g &= f\big((g^{-1}x)_K\big)
    \end{align*}
    for all $x\in A^{\GG}$ and all $g\in \GG$.
\end{knowntheorem}
\noindent Thus the symbol $\varphi(x)_g$ is uniquely determined by the finite pattern $x_{gK}$.

For more on the topological dynamics of one-dimensional CA see the monograph by Kůrka~\cite{Kurka2003a}.
For the group-theoretic aspects of CA, see the monograph by Ceccherini-Silberstein and Coornaert~\cite{ceccherini2010cellular}. 

This paper concerns CA with random initial conditions. Random initial conditions will be prescribed by probability measures on $A^{\GG}$.
A probability measure $\mu$ on $A^\GG$ is said to be \Emph{$\GG$-invariant} if $\mu(g^{-1}E)=\mu(E)$ for every measurable set $E\subseteq A^\GG$ and every $g\in\GG$.
A measurable set $E\subseteq\GG$ is said to be \Emph{$\GG$-invariant} if $g^{-1}E=E$ for each~$g\in\GG$.
A $\GG$-invariant measure $\mu$ is \Emph{$\GG$-ergodic} if for every $\GG$-invariant measurable set $E\subseteq\GG$, we either have $\mu(E)=0$ or $\mu(E)=1$.
Bernoulli measures on $A^\GG$ are examples of $\GG$-ergodic measures. 

\subsection{Percolation}
\label{sec:prelim:percolation}

Let $A$ be a finite set,  $p\colon A\to [0,1]$ be a probability distribution and $G$ a countable set. The \Emph{Bernoulli} measure with marginal~$p$ is the probability measure $\mu_p$ on $A^{G}$ given by
\begin{align*}
    \mu_p([w]) &= \prod_{g\in F} p(w_g)
\end{align*}
for every pattern $w\colon F\to A$ with $F\Subset G$.

Let $\Gamma$ be a graph and $x\in\{\symb{0},\symb{1}\}^{V(\Gamma)}$.  A vertex $v\in V(\Gamma)$ with $x(v)=\symb{1}$ is interpreted as being \Emph{open} in~$x$; otherwise, it is considered \Emph{closed}.  A path in $\Gamma$ is said to be \Emph{open} in~$x$ if all its vertices are open. We say that $x$ \Emph{percolates} in~$\Gamma$ if there exists an infinite open self-avoiding path in~$x$. A probability measure $\mu$ on $\{\symb{0},\symb{1}\}^{V(\Gamma)}$ is said to \Emph{percolate} in $\Gamma$ if 
 \begin{align*}
    \mu\Big(\big\{x \in \{\symb{0},\symb{1}\}^{V(\Gamma)}: \text{$x$ percolates in $\Gamma$}\big\}\Big) &> 0\;.
 \end{align*}



For $p\in[0,1]$, let us consider the Bernoulli measure $\mu_p$ with parameter~$p$ on~$\{\symb{0},\symb{1}\}^\GG$ (i.e., the Bernoulli measure with marginal $\symb{0}\mapsto 1-p$, $\symb{1}\mapsto p$).
By the Kolmogorov zero-one law, for each~$p$, the probability of percolation with respect to $\mu_p$ is either~$0$ or~$1$.
The \Emph{percolation threshold} of $\Gamma$ is defined~as
\begin{align*}
    p_{\critical}(\Gamma) &\isdef
        \sup\big\{p\in [0,1]: \text{$\mu_p$ does not percolate in $\Gamma$}\big\} \;.
\end{align*}
A standard monotonicity argument shows that for every $p>p_\critical(\Gamma)$, the measure $\mu_p$ percolates in~$\Gamma$.
We say that $\Gamma$ has \Emph{non-trivial} percolation threshold if $p_\critical(\Gamma)<1$.

The notions of percolation for configurations of open and closed edges and for measures on such configurations are defined analogously.  To distinguish between them, these two types of percolation are often referred to as \Emph{site} and \Emph{bond} percolation.
A graph with bounded degrees has a non-trivial site percolation threshold if and only if it has a non-trivial bond percolation threshold (see \cite[Propositions~7.10 and~7.11]{LP2016}).
In this paper, we are primarily concerned with site percolation.


It is well-known that, for a finitely generated group~$\GG$, the non-triviality of percolation threshold on the Cayley graphs of $\GG$ does not depend on the choice of the generators (see for instance~\cite[Theorem 7.15]{LP2016}).  Namely, if $S$ and $T$ are two finite generating sets for $\GG$, then $p_{\critical}(\cayley(\GG,S))<1$ if and only if $p_{\critical}(\cayley(\GG,T))<1$.
Thus, we say that $\GG$ has \Emph{non-trivial} percolation threshold if $p_{\critical}(\cayley(\GG,S))<1$ for some (equivalently, for every) choice of the generating set $S$.



Every virtually~$\ZZ$ group is two-ended (see e.g.,~\cite[Corollary 11.34]{Meier2008}) and hence is easily seen to have trivial percolation threshold.
On the other hand, following several partial results by various authors (see the references in~\cite[Section~7.4]{LP2016} and~\cite{DGR+2020}), Duminil-Copin, Goswami, Raoufi, Severo, and Yadin recently proved that every group of super-linear growth (i.e., every group that is not virtually cyclic) has non-trivial percolation threshold~\cite{DGR+2020}, thus establishing the following characterization.

\begin{knowntheorem}[Characterization of groups with non-trivial percolation]
\label{thm:percolation}
    A finitely generated group has a trivial percolation threshold if and only if it is virtually cyclic.
\end{knowntheorem}

For more on percolation theory, see the monographs by Grimmett~\cite{Grimmett1999} and
Lyons and Peres~\cite{LP2016}, 
and the recent survey by Duminil-Copin~\cite{DuminilCopin2019}.

\section{Sensitivity and equicontinuity}
\label{sec:sensitivityandequicontinuity}

In this section, we introduce the notions of sensitivity and equicontinuity with respect to a measure in the setting of CA on groups.



Given a CA $\varphi\colon A^{\GG}\to A^{\GG}$, we define the \Emph{stability set} of a configuration $x\in A^{\GG}$ with respect to a set $F\Subset \GG$ as
\useshortskip
\begin{align*}
    C(x,F,\varphi) &\isdef \{y\in A^{\GG}: \text{$\varphi^n(y)_F = \varphi^n(x)_F$ for all $n\geq 0$}\} \;.
\end{align*} 
This is the set of all configurations whose orbits under $\varphi$ agree with the orbit of $x$ inside the region~$F$.
Let us observe that: 
\begin{observation}
\label{obs:stability-set}
    We have that $y\in C(x,F,\varphi)$ if and only if $C(y,F,\varphi)=C(x,F,\varphi)$.
\end{observation}
\noindent Therefore, ignoring the repetitions, the sets $C(x,F,\varphi)$ partition~$A^\GG$. 

\subsection{Sensitivity}

A CA $\varphi\colon A^\GG\to A^\GG$ is said to be \Emph{sensitive} if there exists a set $F\Subset\GG$ such that, for every configuration $x\in A^\GG$ and every set $E\Subset\GG$, there exists a configuration $y\in A^\GG$ and a time $t\geq 0$ such that $y_E=x_E$ but $\varphi^t(y)_F\neq\varphi^t(x)_F$.
It is not hard to check that $\varphi$ is sensitive if and only if there exists an $F\Subset \GG$ such that $C(x,F,\varphi)$ has empty interior for every $x\in A^{\GG}$.


Gilman introduced the following measurable version of sensitivity, which he called ``almost expansivity''~\cite{Gilman1987}.

\begin{definition}[Sensitivity w.r.t.\ a probability measure]
\label{def:musens}
   Let $\varphi\colon A^{\GG}\to A^{\GG}$ be a CA and $\mu$ a probability measure on $A^{\GG}$.  We say that $\varphi$ is \Emph{sensitive with respect to $\mu$} (\Emph{$\mu$-sensitive} for short) if there exists a finite set $F\Subset \GG$ such that $\mu\big(C(x,F,\varphi)\big)=0$ for $\mu$-almost every~$x\in A^{\GG}$.
\end{definition}

In analysis and dynamical systems theory, analogies between topological and measurable concepts are often insightful, enabling the translation of results from one framework to the other and the formulation of reasonable conjectures. For a survey of such interactions, see the monograph by Oxtoby~\cite{Oxtoby1996} and the article by Glasner and Weiss~\cite{GW2006}.

\begin{remark}
\label{rem:musens:sure}
    The ``almost sure'' quantifier in the definition of $\mu$-sensitivity can be replaced with a ``sure'' quantifier; that is, a CA $\varphi:A^\GG\to A^\GG$ is $\mu$-sensitive if and only if there exists an $F\Subset\GG$ such that $\mu\big(C(x,F,\varphi)\big)=0$ for every $x\in A^\GG$.
    Indeed, from Observation~\ref{obs:stability-set} it follows that if $\mu\big(C(x,F,\varphi)\big)>0$ for some $x$, then $\mu\big(C(y,F,\varphi)\big)>0$ for every $y\in C(x,F,\varphi)$, hence $\varphi$ is not $\mu$-sensitive.
\end{remark}


    
 

Any CA that is sensitive with respect to a fully supported measure is also sensitive. The converse does not hold. Gilman~\cite[Section~3]{Gilman1987} described an example that is sensitive and not $\mu$-sensitive for some, but not all, Bernoulli measures~$\mu$. Here, we present a more extreme example, a sensitive CA that is not $\mu$-sensitive for any Bernoulli measure~$\mu$.

\begin{example}[Sensitive but not $\mu$-sensitive]
\label{exp:sens-but-not-mu-sens}
Consider the CA $\varphi_\pine\colon \{\symb{0},\symb{1}\}^{\ZZ} \to \{\symb{0},\symb{1}\}^{\ZZ}$, where 
\begin{align*}
    \varphi_\pine(x)_n &\isdef
        \begin{cases}
            \symb{1}        & \text{if $x_{n+1}=x_{n+2}=\symb{1}$,} \\
            \symb{0}        & \text{otherwise,}
        \end{cases}
\end{align*}
for every $x\in\{\symb{0},\symb{1}\}^\ZZ$ and $n\in\ZZ$.
We call this the \Emph{pine processionary CA} in analogy with the behaviour of the pine processionary caterpillars. In this model, the symbol~$\symb{1}$ represents a caterpillar. The caterpillars form connected chains or ``processions'' that move to the left. At each time step, the last caterpillar in the chain (the rightmost one) detaches and disappears, as if it has become lost. An illustration of a space-time diagram of this CA is shown in Figure~\ref{fig:cuncunas}.

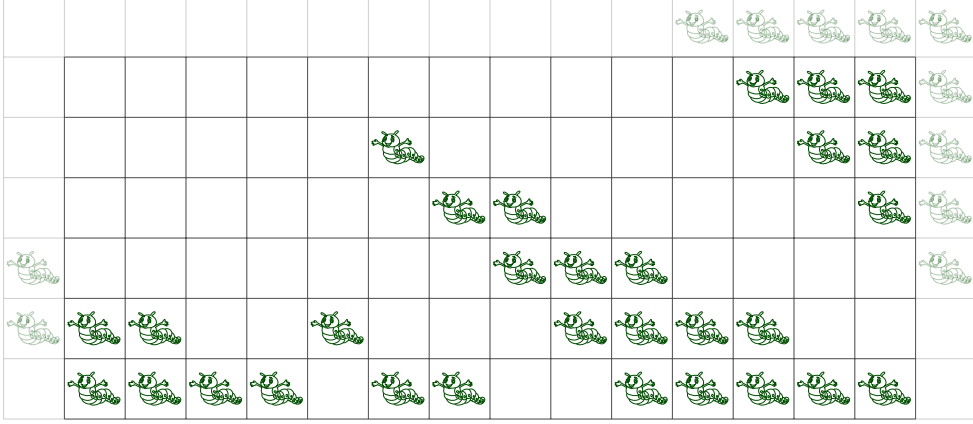
\begin{figure}[ht!]
    \centering
    \scalebox{0.8}{
    \begin{tikzpicture}
        \begin{scope}[shift = {(0.5,-0.5)}]
            \draw[opacity =0.2] (-8,0) grid (8,7);
            \draw[opacity =0.6] (-7,0) grid (7,6);
        \end{scope}
        \foreach \i in {-6,-5,-4,-3, -1,0, 3,4,5,6,7}{
            \node[opacity =1] at (\i,0) {$\scalebox{0.1}{\fcCaterpillar{0.5}{green!30!black}{2}}$};
        }
        \foreach \i in {-6,-5, -2, 2,3,4,5}{
            \node[opacity =1] at (\i,1) {$\scalebox{0.1}{\fcCaterpillar{0.5}{green!30!black}{2}}$};
        }
         \node[opacity =0.3] at (-7,1) {$\scalebox{0.1}{\fcCaterpillar{0.5}{green!30!black}{2}}$};
        \foreach \i in {1,2,3}{
            \node[opacity =1] at (\i,2) {$\scalebox{0.1}{\fcCaterpillar{0.5}{green!30!black}{2}}$};
        }
        \node[opacity =0.3] at (-7,2) {$\scalebox{0.1}{\fcCaterpillar{0.5}{green!30!black}{2}}$};
        \node[opacity =0.3] at (8,2) {$\scalebox{0.1}{\fcCaterpillar{0.5}{green!30!black}{2}}$};
        \foreach \i in { 0,1,7}{
            \node[opacity =1] at (\i,3) {$\scalebox{0.1}{\fcCaterpillar{0.5}{green!30!black}{2}}$};
        }
        \node[opacity =0.3] at (8,3) {$\scalebox{0.1}{\fcCaterpillar{0.5}{green!30!black}{2}}$};
        \foreach \i in { -1,6,7}{
            \node[opacity =1] at (\i,4) {$\scalebox{0.1}{\fcCaterpillar{0.5}{green!30!black}{2}}$};
        }
        \node[opacity =0.3] at (8,4) {$\scalebox{0.1}{\fcCaterpillar{0.5}{green!30!black}{2}}$};
        \foreach \i in { 5,6,7}{
            \node[opacity =1] at (\i,5) {$\scalebox{0.1}{\fcCaterpillar{0.5}{green!30!black}{2}}$};
        }
        \node[opacity =0.3] at (8,5) {$\scalebox{0.1}{\fcCaterpillar{0.5}{green!30!black}{2}}$};
        \foreach \i in { 4,5,6,7,8}{
            \node[opacity =0.3] at (\i,6) {$\scalebox{0.1}{\fcCaterpillar{0.5}{green!30!black}{2}}$};
        }
    \end{tikzpicture}
    }
    \caption{A section of the space-time diagram of the pine processionary CA with time going upwards. The symbol $\symb{1}$ has been replaced by a cute caterpillar and $\symb{0}$ by an empty space.}
    \label{fig:cuncunas}
\end{figure}


We first note that $\varphi_\pine$ is sensitive.  Indeed, given $x\in\{\symb{0},\symb{1}\}^\ZZ$ and $n\in\NN$, we can construct two configurations $y,z\in\{\symb{0},\symb{1}\}^\ZZ$, where
\begin{align*}
    y_k &\isdef
        \begin{cases}
            x_k         & \text{if $-n\leq k\leq n$,} \\
            \symb{0}    & \text{otherwise,}
        \end{cases}
    &
    z_k &\isdef
        \begin{cases}
            x_k         & \text{if $-n\leq k\leq n$,} \\
            \symb{1}    & \text{otherwise.}
        \end{cases}
\end{align*}
It is clear that $\varphi_\pine^{n+1}(y)_0=\symb{0}$ and $\varphi_\pine^{n+1}(z)_0=\symb{1}$.  Hence, either $\varphi_\pine^{n+1}(y)_0\neq\varphi_\pine^{n+1}(x)_0$ or $\varphi_\pine^{n+1}(z)_0\neq\varphi_\pine^{n+1}(x)_0$.  Therefore, $\varphi_\pine$ is sensitive with $F\isdef\{0\}$ as witness.

Let $0\leq q \leq 1$ and consider the Bernoulli measure $\mu_q$ with parameter $q$ on $\{\symb{0},\symb{1}\}^\ZZ$.
Let us show that $\varphi_\pine$ is not sensitive with respect to $\mu_q$. 
Clearly, $\varphi_\pine$ is not $\mu_1$-sensitive. Thus, assume $q<1$.
For $n\in \NN$ and $i \in \ZZ$, we define 
\begin{align*}
    E_n^{i} &\isdef \{x\in \{\symb{0},\symb{1}\}^{\ZZ}: \varphi_\pine^n(x)_i= \symb{1}\} \\
    &= \big\{x\in \{\symb{0},\symb{1}\}^{\ZZ}: \text{$x_j=1$ for all $j \in \{i+n,\dots,i+2n\}$}\big\} \;.
\end{align*}
Clearly $\mu_q(E_n^{i})=q^{n}$ and hence
\useshortskip
\begin{align*}
    \sum_{n=1}^{\infty}\mu_q(E_n^{i}) &< \infty \;.
\end{align*}
Using the Borel-Cantelli lemma, it follows that the measure of the set of configurations that are in infinitely many $E^i_n$'s is zero.
In other words, for almost every $x\in\{\symb{0},\symb{1}\}^\ZZ$, there exists an $m\in\NN$ such that $\varphi_\pine^n(x)_i=\symb{0}$ for all $n\geq m$.

Now, let $F\Subset \GG$.  For $m \in \NN$, we define  
\begin{align*}
    Z^F_m &\isdef
    \big\{x\in \{\symb{0},\symb{1}\}^{\ZZ}: \text{$\varphi_\pine^n(x)_i=\symb{0}$ for all $i\in F$ and $n\geq m$}\big\} \;.
\end{align*}
From the above observation, it follows that $\mu_q(Z_m^F)>0$ for some $m \in \NN$. 
Since there are only finitely many possibilities for the values $x_F,\varphi_\pine(x)_F,\ldots,\varphi_\pine^{m-1}(x)_F$, it follows that for some $w^{(0)}_F, w^{(1)}_F, \ldots, \allowbreak w^{(m-1)}_F\in\{\symb{0},\symb{1}\}^F$, the set
\begin{align*}
    D^F_m \isdef D^F_m(w^{(0)}_F,w^{(1)}_F,\ldots,w^{(m-1)}_F)
    &\isdef
    \big\{x\in Z^F_m: \text{$\varphi_\pine^n(x)_F=w^{(n)}_F$ for $0\leq n\leq m-1$}\big\}
\end{align*}
has positive measure.
For every $x\in D^F_m$, we have $C(x,F,\varphi_\pine)\supseteq D^F_m$, hence $\mu_q\big(C(x,F,\varphi_\pine)\big)> 0$.  Therefore, $\varphi_\pine$ is not $\mu_q$-sensitive.
\end{example}



\subsection{Equicontinuity}

An \Emph{equicontinuity point} of a CA $\varphi\colon A^\GG\to A^\GG$ is a configuration $x\in A^\GG$ such that, for every $F\Subset\GG$, there exists $E\Subset\GG$ such that every $y\in A^\GG$ with $y_E=x_E$ satisfies $\varphi^t(y)_F=\varphi^t(x)_F$ for all $t\geq 0$.
The latter condition is equivalent to the equicontinuity of the family $(\varphi^n)_{n\in\NN}$ at $x$ in the sense of analysis, hence the name.

We say that a CA $\varphi$ is \Emph{almost equicontinuous} if its equicontinuity points form a residual subset (i.e. it contains a countable intersection of dense open subsets) of~$A^\GG$.  

\begin{proposition}[Characterization of almost equicontinuity]
\label{prop:charac_almost}
    A CA $\varphi\colon A^{\GG}\to A^{\GG}$ is almost equicontinuous if and only if the set of configurations $x\in A^{\GG}$ for which $C(x,F,\varphi)$ has non-empty interior for every $F\Subset \GG$ is residual.
\end{proposition}

\begin{proof}
    First, observe that a configuration $x\in A^{\GG}$ is an equicontinuity point of $\varphi$ if and only if for every $F\Subset\GG$, the set $C(x,F,\varphi)$ has $x$ in its interior.
    
    Let $F\Subset\GG$ be fixed.
    Recall from Observation~\ref{obs:stability-set} that the family $\family{C}\isdef\{C(x,F,\varphi): x\in A^\GG\}$ partitions~$A^\GG$.  Let $\family{D}\isdef\{D\in \family{C}: \interior{D}\neq\varnothing\}$, where $\interior{D}$ denotes the interior of~$D$.
    Note that every element of~$\family{C}$ is closed.
    Let
    \begin{align*}
        E_F &\isdef \{x\in A^{\GG}: \interior{C}(x,F,\varphi)\ni x\} =
            \bigcup_{D\in\family{D}}\interior{D} \;, \\
        E'_F &\isdef \{x\in A^{\GG}: \interior{C}(x,F,\varphi)\neq\varnothing\} =
            \bigcup_{D\in\family{D}} D 
            = E_F \cup \bigcup_{D\in\family{D}}\partial D \;,
    \end{align*}
    where $\partial D$ denotes the boundary of~$D$.
    A topological space that has a countable basis can have at most countably many disjoint open sets, hence $\family{D}$ is at most countable.
    Since the boundary of every closed set is nowhere dense, it follows that $\bigcup_{D\in\family{D}}\partial D$ is a meagre set.
    Therefore, $E_F$ is residual if and only if $E'_F$ is.
    Intersecting over all $F\Subset\GG$, we obtain that $\varphi$ is almost equicontinuous if and only if $\bigcap_{F\Subset\GG} E'_F$ is residual.
    This proves the claim.
\end{proof}

Gilman proposed a measurable analogue of almost equicontinuity for CA on~$\ZZ$ in terms of density points~\cite{Gilman1987}. We present an alternative definition, that naturally aligns with Proposition~\ref{prop:charac_almost}, and we demonstrate its equivalence to Gilman's definition. 

\begin{definition}[Equicontinuity w.r.t.\ a probability measure]
\label{def:mueq}
    Let $\varphi\colon A^{\GG}\to A^{\GG}$ be a CA and $\mu$ a probability measure on $A^{\GG}$. We say that $\varphi$ is \Emph{equicontinuous with respect to $\mu$} (\Emph{$\mu$-equicontinuous} for short) if for  $\mu$-almost every $x \in A^{\GG}$, we have $\mu(C(x,F,\varphi))>0$ for every $F\Subset\GG$.
\end{definition}

\begin{remark}
\label{rem:mueq:quantifiers}
Since the family of finite subsets of $\GG$ is countable,
the order of quantifiers in the definition of $\mu$-equicontinuity can be reversed; that is, a cellular automaton $\varphi\colon A^{\GG}\to A^{\GG}$ is $\mu$-equicontinuous if and only if, for every finite set $F \Subset \GG$, we have $\mu(C(x,F,\varphi)) > 0$ for $\mu$-almost every~$x \in A^{\GG}$.
\end{remark}

A chain $J_1\subseteq J_2\subseteq\cdots$ of finite subsets of $\GG$ is said to be \Emph{co-final} if $\bigcup_{n=1}^\infty J_n=\GG$.  For instance, the balls $B_n(e)$ in the Cayley graph of a finitely generated group $\GG$ form a co-final chain.
Let $\mu$ be a probability measure on~$A^\GG$, $E\subseteq A^\GG$ a measurable set, and $(J_n)_{n\in \NN}$ a co-final chain of finite subsets of~$\GG$.  We say that a configuration $x\in A^\GG$ is a \Emph{point of $\mu$-density of $E$ with respect to~$(J_n)_{n\in\NN}$} if $x$ is in the topological support of~$\mu$ and
\begin{align*}
    \lim_{n\to\infty} \mu\big(E \given[\big] [x_{J_n}]\big) = \lim_{n\to\infty} \frac{\mu(E\cap [x_{J_n}])}{\mu([x_{J_n}])} &= 1 \;.
\end{align*}
%
%
%
%
The following theorem is the analogue of Lebesgue's density theorem for the Cantor set. 
It follows directly from Levy's zero-one law~\cite[Theorem~5.5.8]{Durrett2010}.
\begin{knowntheorem}[Points of density]
\label{thm:lebesgue-density:Cantor}
    Let 
    $(J_n)_{n\in\NN}$ be a co-final chain of finite subsets of~$\GG$.
    Let $\mu$ be a probability measure on $A^\GG$ and $E\subseteq A^\GG$ a measurable set.  Then, $\mu$-almost every configuration in $E$ is a point of $\mu$-density of~$E$ with respect to~$(J_n)_{n\in\NN}$.
\end{knowntheorem}

The following proposition shows that $\mu$-equicontinuity can be equivalently defined in terms of density points with respect to a co-final chain~$(J_n)_{n\in\NN}$.

\begin{proposition}[Characterization of $\mu$-equicontinuity]
\label{prop:mu-equicontinuity}
    Let $(J_n)_{n\in\NN}$ be a co-final chain of finite subsets of~$\GG$.
    Let 
    $\varphi\colon A^{\GG}\to A^{\GG}$ be a CA and $\mu$ a probability measure on $A^{\GG}$. Then, $\varphi$ is $\mu$-equicontinuous if and only if
    for $\mu$-almost every $x\in A^\GG$ and every $F\Subset\GG$, we have
    \begin{align*}
        \lim_{n\to\infty} \mu\big(C(x,F,\varphi) \given[\big] [x_{J_n}]\big) &= 1 \;,
    \end{align*}
    that is, $x$ is a point of $\mu$-density of $C(x,F,\varphi)$ with respect to $(J_n)_{n\in\NN}$.
\end{proposition}
\begin{proof}
    Suppose that for $\mu$-almost every~$x\in A^\GG$ and every $F\Subset\GG$, we have $\mu\big(C(x,F,\varphi)\big)>0$.
    Let $F\Subset\GG$ be fixed.
    Recall from Observation~\ref{obs:stability-set} that the family $\family{C}\isdef\{C(x,F,\varphi): x\in A^\GG\}$ partitions~$A^\GG$.  Let $\family{D}\isdef\{D\in \family{C}: \mu(D)>0\}$.
    A probability space can have at most countably many disjoint sets of positive measure, hence $\family{D}$ is at most countable.
    For $D\in\family{D}$, let
    \begin{align*}
        \widetilde{D} &\isdef \big\{x\in D: \text{$x$ is a point of $\mu$-density of~$C(x,F,\varphi)=D$ with respect to $(J_n)_{n\in\NN}$}\big\} \;.
    \end{align*}
    By Theorem~\ref{thm:lebesgue-density:Cantor},
    $\mu(\widetilde{D})=\mu(D)$ for each $D$.
    Therefore,
    \begin{align*}
        \mu\Big(\bigcup_{D\in\family{D}}\widetilde{D}\Big) &= \sum_{D\in\family{D}}\mu(\widetilde{D})
        =
        \sum_{D\in\family{D}}\mu(D)
        =
        \mu\Big(\bigcup_{D\in\family{D}}D\Big)
        = 1 \;.
    \end{align*}
    In other words, the set $M_F$ of configurations $x\in A^\GG$ such that $x$ is a point of $\mu$-density of~$C(x,F,\varphi)$ 
    has $\mu$-measure~$1$.
    Since the family of finite subsets of $\GG$ is countable, we conclude that $\mu\big(\bigcap_{F\Subset\GG}M_F\big)=1$, proving the forward implication.
    %

    The converse follows from the fact that, by definition, a null set cannot have a point of density.
\end{proof}

One can also characterize $\mu$-equicontinuity using a condition obtained from Lusin's characterization of Borel functions~\cite{GarciaRamos2017}.

The following example shows that a CA that is equicontinuous with respect to a Bernoulli measure need not be almost equicontinuous.
\begin{example}[$\mu$-equicontinuous but not almost equicontinuous]
Consider the pine processionary CA~$\varphi_\pine$ of Example~\ref{exp:sens-but-not-mu-sens} and a Bernoulli measure~$\mu_q$.
The argument used to show the sensitivity of~$\varphi_\pine$ also shows that $\varphi_\pine$ has no equicontinuity points.
On the other hand, since $\varphi_\pine$ is not $\mu_q$-sensitive, by Gilman's theorem (Theorem~\ref{thm:gilman}), it has to be $\mu_q$-equicontinuous.
\end{example}


\section{Gilman's dichotomy}
\label{sec:Gilmandichotomy}

From the definitions, it is clear that a CA cannot be both sensitive and equicontinuous with respect to a measure~$\mu$.
Inspired by Gilman's result (Theorem~\ref{thm:gilman}), we introduce the following terminology.

\begin{definition}
    We say that a group $\GG$ satisfies \Emph{Gilman's dichotomy} if for every finite set $A$ and every Bernoulli measure $\mu$ on $A^\GG$, every CA $\varphi\colon A^{\GG}\to A^{\GG}$ is either $\mu$-sensitive or $\mu$-equicontinuous.
\end{definition}

As mentioned in the introduction, the objective of this paper is to characterize which groups satisfy Gilman's dichotomy.

\subsection{Generalization of Gilman's dichotomy to virtually $\ZZ$ groups}

We begin by extending Gilman's result to virtually $\ZZ$ groups. In fact, we will show that, on such groups, Gilman's dichotomy holds not only for Bernoulli measures, but for all $\GG$-ergodic measures.
Our proof relies on the well-known Freudenthal-Hopf characterization of virtually $\ZZ$ groups as those that have two ends (see for instance~\cite[Corollary 11.34]{Meier2008}).
This means that for every generating set $S$ of $\GG$, if we remove a sufficiently large ball from $\cayley(\GG,S)$, the resulting graph contains exactly two infinite connected components. 

\begin{lemma}\label{lema:tecnicotrivial}
    Let $\GG$ be a virtually~$\ZZ$ group and $K\Subset\GG$.
    Let $h\in\GG$ be a non-torsion element of~$\GG$ and let $W\Subset \GG$ be a finite set such that $\generatedby{h} W = \GG$.
    
    There exists a set $F_0\Subset\GG$ such that for every $F\Subset\GG$, there exist $n_0\in\NN$ such that for all $n,m> n_0$ and $u,v \in W$, the set $\GG\setminus(h^{-n}uF_0\cup h^mv F_0)$ can be partitioned into two sets $\GG_\xint$ and $\GG_\xext$ such that \begin{enumerate}[label={\textup{(\roman*)}}]
        \item \label{item:tecnicotrivial:finite-xint}
            $\GG_\xint$ is finite,
        \item \label{item:tecnicotrivial:no-info-leak}
            $\GG_\xint K\cap\GG_\xext=\varnothing$, and
        \item \label{item:tecnicotrivial:F-contained}
            $\GG_\xext \cap F = \varnothing$.
    \end{enumerate}
\end{lemma}
\begin{proof}
    Fix a finite set $S$ of generators for $\GG$ such that $K \subseteq S$.  As usual, let $B_n \isdef B_n(e)$ be the centered ball of radius $n$ in $\cayley(\GG,S)$. Let $r\geq 1$ be such that $W\subseteq B_r$ and such that after removing all vertices in $B_{r}$ the graph $\cayley(\GG,S)$ has two infinite connected components $\GG_{\ell}$ and $\GG_{r}$. 

    Notice that there exists some $t_0$ such that for every $t\geq t_0$, then $h^{-t}B_r$ and $h^{t}B_r$ are contained in the two distinct infinite connected components. Up to renaming the sets, we have $h^{-t}B_r \subseteq \GG_{\ell}$ and $h^tB_r \subseteq \GG_{r}$.

    We define $F_0 = \bigcup_{w \in W} w^{-1} B_r$. Now fix $F\Subset \GG$. As $W \subseteq B_r$, there exists $\kappa_0 \in \ZZ$ such that
    \begin{align*}
        F &\subseteq \bigcup_{i = -\kappa_0}^{\kappa_0} h^{-i}B_r \;.
    \end{align*}
    Now let $n_0 \isdef t_0+\kappa_0$, and consider $n,m > n_0$ and $u,v \in W$. By the construction of $F_0$, we have $h^{-n}B_r \subseteq h^{-n}uF_0$ and $h^mB_r \subseteq h^mvF_0$. In particular, $\GG\setminus h^{-n}uF_0$ has two infinite connected components, which we call $\GG^{-n}_{\ell}$ and $\GG^{-n}_{r}$, and they satisfy
    \begin{align*}
        \GG^{-n}_{\ell} \subseteq h^{-n}\GG_{\ell}
        \qquad\text{and}\qquad
        \GG^{-n}_{r} \subseteq h^{-n}\GG_{r} \;.
    \end{align*}
    Similarly,  $\GG\setminus h^mvB_r$ has two infinite connected components $\GG^m_{\ell}$ and $\GG^m_{r}$, and they satisfy
    \begin{align*}
        \GG^m_{\ell} \subseteq h^m\GG_{\ell}
        \qquad\text{and}\qquad
        \GG^m_{r} \subseteq h^m\GG_{r} \;.
    \end{align*}
    We take $\GG_{\xext} \isdef (\GG^{-n}_{\ell} \cap \GG^m_{\ell}) \cup (\GG^{-n}_{r}\cap\GG^m_{r} )$ and notice that its complement in $\GG$ is a finite set. Finally, set $\GG_{\xint}$ such that $\GG_{\xint}$, $\GG_{\xext}$ and $(h^{-n}uF_0\cup h^mv F_0)$ partition $\GG$.
    
    Now, property~\ref{item:tecnicotrivial:finite-xint} is direct from the definition of $\GG_{\xint}$. Property~\ref{item:tecnicotrivial:no-info-leak} follows from the requirement that $K\subseteq S$. Finally, as $n,m \geq n_0$, we have 
    \begin{align*}
        h^{i}B_r \subseteq h^{-n}\GG_{r}
        \qquad\text{and}\qquad
        h^iB_r \subseteq h^m\GG_{\ell}
    \end{align*}
    for every $i \in \{-\kappa_0,\dots,\kappa_0\}$.
    In particular, this implies that
    \begin{align*}
        h^{i}B_r \cap \GG^{-n}_{\ell} \subseteq h^{i}B_r \cap h^{-n} \GG_{\ell} = \varnothing
        \qquad\text{and}\qquad
        h^iB_r \cap \GG^m_{r} \subseteq h^iB_r \cap h^m\GG_{r} = \varnothing \;,
    \end{align*}
    from which we obtain~\ref{item:tecnicotrivial:F-contained}.
\end{proof}

The second ingredient of our proof is the Poincaré recurrence theorem for a single measure-preserving transformation~\cite[Theorem~1.4]{Walters1982}. (For other versions regarding group actions, see~\cite[Theorem 2.10]{KL2017}).
We give a short proof for the sake of completeness.



\begin{knowntheorem}[Poincaré recurrence theorem]
\label{thm:poincare}
	Let $\psi:X\to X$ be a measurable map on a measurable space~$X$ and let $\mu$ be a $\psi$-invariant probability measure on~$X$.
	For every measurable set $E\subseteq X$, we have $\mu(E')=\mu(E)$, where
	\begin{align*}
		E' &= \big\{ x \in E: \textup{$\psi^n(x) \in E$ for infinitely many $n>0$}\big\} \;.
	\end{align*}
\end{knowntheorem}
\begin{proof}
    Let $T\Subset \NN$ and $A_T \isdef \{x \in X : \text{for each $n\geq 0$, $n\in T$ if and only if $\psi^n(x)\in E$}\}$.  Note that the sets $\psi^{-m}(A_T)$ (for $m\geq 0$) are disjoint. Since $\mu$ is $\psi$-invariant, we have
    \begin{align*}
        1 \geq \mu\left(\bigcup_{m \geq 0} \psi^{-m}(A_T) \right) &= \sum_{m \geq 0} \mu\big(\psi^{-m}(A_T)\big) = \sum_{m \geq 0}\mu(A_T) \;,
    \end{align*}
    thus $\mu(A_T)=0$.
    As $E\setminus E' \subseteq \bigcup_{T \Subset \NN}A_T$, it follows that $\mu(E')=\mu(E)$.
\end{proof}

\begin{theorem}[Generalization of Gilman's dichotomy]
\label{thm:virtually_Z_satisfies_dichotomy}
	Let $\GG$ be a virtually $\ZZ$ group, $\mu$ a $\GG$-ergodic measure on $A^{\GG}$ and $\varphi \colon A^{\GG}\to A^{\GG}$ a CA. Then, $\varphi$ is either $\mu$-sensitive or $\mu$-equicontinuous. In particular, every virtually $\ZZ$ group satisfies Gilman's dichotomy. 
\end{theorem}

\begin{proof}
	By Theorem~\ref{thm:curtis}, there exists a set $K \Subset \GG$ and local map $f\colon A^{K}\to A$ such that $\varphi(x)_g = f\big((g^{-1}x)_K\big)$ for every $g \in \GG$.
	Let $h\in\GG$ be a non-torsion element of~$\GG$ such that $\generatedby{h}$ is a normal finite index subgroup of $\GG$, and let $W\Subset \GG$ be such that $\generatedby{h}W = \GG$ and $e \in W$.
	Let $F_0\Subset\GG$ be as in Lemma~\ref{lema:tecnicotrivial}.
	
	Suppose that $\varphi$ is not $\mu$-sensitive.
	By definition, there exists a configuration $z\in A^\GG$ such that $\mu\big(C(z,F_0,\varphi)\big)>0$.
	We claim that, $\mu\big(C(x,F,\varphi)\big)>0$ for every $F\Subset\GG$ and $\mu$-almost every $x\in A^\GG$, which means $\varphi$ is $\mu$-equicontinuous.
   
	Let $(J_n)_{n\in\NN}$ be an arbitrary co-final chain of finite subsets of~$\GG$.
	Clearly, for each $g$, the sequence $(gJ_n)_{n \in \NN}$ is also co-final in~$\GG$.
	Let $D\isdef\bigcap_{g\in\GG}D_g$, where
	\begin{align*}
		D_g &\isdef \{ x \in C(z,F_0,\varphi) :
			\text{$x$ is a point of density of $C(z,F_0,\varphi)$ with respect to $(gJ_n)_{n\in\NN}$}
		\} \;.
	\end{align*}
	By Theorem~\ref{thm:lebesgue-density:Cantor}, we have $\mu(D_g)=\mu\big(C(z,F_0,\varphi)\big)$ for each $g\in\GG$.  Hence, $\mu(D)=\mu\big(C(z,F_0,\varphi)\big)>0$.
	
	We say that a set $T\subseteq\ZZ$ is \Emph{bi-infinite} if $T\neq\varnothing$ and for every $n\in T$, there exist $\ell,r\in T$ such that $\ell<n<r$.
	Let
	\begin{align*}
		Y &\isdef \big\{x\in\supp(\mu): \text{$\{n\in\ZZ: h^nx\in WD\}$ is bi-infinite}\big\}
	\end{align*}
	be the set of all configurations in the topological support of~$\mu$ whose orbits under $h$ visit $WD$ infinitely many times in both directions.
	Let us show that $\mu(Y)=1$.
	Since $W\ni e$, we have $Y\supseteq D^+\cap D^-$, where
	\begin{align*}
		D^+ &\isdef \{x\in D: \text{$h^{-n}x\in D$ for infinitely many $n> 0$}\} \subseteq D \;, \\
		D^- &\isdef \{x\in D: \text{$h^{m}x\in D$ for infinitely many $m> 0$}\} \subseteq D \;.
	\end{align*}
	Applying Poincaré's recurrence theorem (Theorem~\ref{thm:poincare}) to translations by $h$ and $h^{-1}$, we obtain that $\mu(D^+)=\mu(D^-)=\mu(D)$, hence $\mu(Y)\geq \mu(D^+\cap D^-)=\mu(D)$.
	Let us verify that $Y$ is $\GG$-invariant.  The $\GG$-ergodicity of $\mu$ would then imply that $\mu(Y)=1$.
	Let $g\in\GG$.  By the normality of~$\generatedby{h}$, there exists an integer $k\neq 0$ such that $hg=gh^k $.
	Furthermore, since $\generatedby{h}W=\GG$, for every $w\in W$, there exists a $w'\in W$ and an integer $i(w)$ such that $g^{-1}w=h^{i(w)}w'$.  Set $I\isdef\{i(w): w\in W\}$.
	Now, let $x\in g^{-1}Y$.  We have $gx\in Y$, which means the set $\{n\in\ZZ: h^n gx\in WD\}$ is bi-infinite.
	But $h^n gx\in WD$ implies that $h^{-i+kn}x\in WD$ for some $i\in I$.  Since $I$ is finite and $k\neq 0$, it follows that the set $\{m\in\ZZ: h^m x\in WD\}$ is also bi-infinite, which means $x\in Y$.  Therefore, $g^{-1}Y\subseteq Y$ for every $g\in\GG$, and hence $Y$ is $\GG$-invariant.

	Now, let $F\Subset\GG$ and $x\in Y$.  We prove that $\mu\big(C(x,F,\varphi)\big)>0$.
	By the choice of $F_0$, there exists an $n_0\in\NN$ (corresponding to~$F$, as in Lemma~\ref{lema:tecnicotrivial}) such that for all $n,m> n_0$ and $u,v \in W$, the set $\GG\setminus(h^{-n}uF_0\cup h^mv F_0)$ can be partitioned into two sets $\GG_\xint$ and $\GG_\xext$ such that
	\begin{enumerate*}[label={(\roman*)}]
		\item $\GG_\xint$ is finite,
		\item $\GG_\xint K\cap\GG_\xext=\varnothing$, and
		\item $\GG_\xext \cap F = \varnothing$.
	\end{enumerate*}
	Since $x\in Y$, there exist $n,m>n_0$ and $u,v\in W$ be such that $u^{-1}h^nx, v^{-1}h^{-m}x\in D$.
	Let $\GG_\xint$ and $\GG_\xext$ be as above.
	
	Since $u^{-1}h^nx$ is a point of density of $C(z,F_0,\varphi)$ with respect to $(u^{-1}h^nJ_n)_{n\in\NN}$, the configuration~$x$ is a point of density of $A^-\isdef h^{-n}u C(z,F_0,\varphi)=C(h^{-n}uz,h^{-n}uF_0,\varphi)$ with respect to $(J_n)_{n\in\NN}$.
	Likewise, $x$ is a point of density of $A^+\isdef h^m vC(z,F_0,\varphi)=C(h^m vz,h^m vF_0,\varphi)$ with respect to $(J_n)_{n\in\NN}$.
	Pick $t\in\NN$ large enough such that
	\begin{itemize}
		\item $\GG_\xint\subseteq J_t$,
		\item $\mu\big([x_{J_t}]\cap A^-\big) > \frac{1}{2}\mu\big([x_{J_t}]\big)$, and
		\item $\mu\big([x_{J_t}]\cap A^+\big) > \frac{1}{2}\mu\big([x_{J_t}]\big)$.
	\end{itemize}
	Note that the latter two conditions, along with the fact that $x$ is in the topological support of~$\mu$, imply that $\mu\big([x_{J_t}]\cap A^-\cap A^+\big)>0$.
	
	We claim that
	\begin{align}
    \label{eq:gilman:virtually-Z:proof}
		[x_{J_t}]\cap A^-\cap A^+ &\subseteq C(x,\GG\setminus\GG_\xext,\varphi)\subseteq C(x,F,\varphi);
	\end{align}
	from which it follows that $\mu\big(C(x,F,\varphi)\big)\geq\mu\big([x_{J_t}]\cap A^-\cap A^+\big) >0$.

	The second inclusion in~\eqref{eq:gilman:virtually-Z:proof} follows trivially from fact that $F\cap\GG_\xext=\varnothing$.
	To prove the first inclusion, let $y\in [x_{J_t}]\cap A^-\cap A^+$, and suppose on the contrary that $y\notin C(x,\GG\setminus\GG_\xext,\varphi)$.
	Let $n\geq 0$ be the smallest integer such that $\varphi^n(y)_g\neq\varphi^n(x)_g$ for some $g\in\GG\setminus\GG_\xext=\GG_\xint\cup h^{-n}uF_0 \cup h^m vF_0$.	
	First, note that $g$ cannot belong to $h^{-n}uF_0$ because $x,y\in A^-=C(h^{-n}uz,h^{-n}uF_0,\varphi)$.
	Similarly, $g$~cannot belong to $h^m vF_0$. 
	Thus, we must have $g\in\GG_\xint$.
	The minimality of $n$ and the fact that $gK\cap\GG_\xext\subseteq\GG_\xint K\cap \GG_\xext=\varnothing$ imply that $\varphi^{n-1}(x)_{gK}=\varphi^{n-1}(x)_{gK}$.  But this gives
	\begin{align*}
		\varphi^n(y)_g = f\Big(\big(g^{-1}\varphi^{n-1}(y)\big)_{K}\Big) &=
			f\Big(\big(g^{-1}\varphi^{n-1}(x)\big)_{K}\Big)  =  \varphi^n(x)_g \;,
	\end{align*}
	which contradicts the assumption.
	Thus the second inclusion in~\eqref{eq:gilman:virtually-Z:proof} also holds.
	
	We have shown that $\mu\big(C(x,F,\varphi)\big)>0$ for every $F\Subset\GG$ and every $x\in Y$, where $Y$ is a measurable set with $\mu(Y)=1$, which means $\varphi$ is $\mu$-equicontinuous.
	This concludes the proof of the theorem.
\end{proof}

\subsection{Percolated additive CA}

We show that unless a finitely generated group $\GG$ has a trivial percolation threshold, the percolated additive CA, with a suitable choice of the set of generators, is neither equicontinuous nor sensitive with respect to the uniform Bernoulli measure.  Combined with Theorem~\ref{thm:virtually_Z_satisfies_dichotomy} and Theorem~\ref{thm:percolation}, this proves Theorem~\ref{thm:main-result}.

We recall from the introduction the definition of the {percolated additive CA}. Let $S\Subset \GG$ be a generating set for $\GG$, and $A\isdef\{\symb{0},\symb{1}\}\times\{\symb{0},\symb{1}\}^S$. The \Emph{percolated additive CA} on $\GG$ associated to $S$ is the CA defined by the map $\varphi\colon A^\GG\to A^\GG$, where
    \begin{align*}
        \varphi(x,w)_g &\isdef
            \bigg(\Big(\sum_{s\in S} w_{g}(s)\cdot x_{gs}\Big) \bmod{2}, w_g\bigg),
            \qquad\text{for every $g \in \GG$.}
    \end{align*}

\begin{observation}
\label{rem:percolated additive:additivity}
    A percolated additive CA is linear in the first component, in the sense that 
    \begin{align*}
        \varphi(x+y,w) &= \varphi(x,w)+ \varphi(y,w)
    \end{align*}
    for every $w \in \{\symb{0},\symb{1}\}^{S\times \GG}$ and $x,y \in \{\symb{0},\symb{1}\}^{\GG}$.
\end{observation}

Given $w\in\{\symb{0},\symb{1}\}^{\GG\times S}$, we let $\varphi_w(x)$ denote the first component of $\varphi(x,w)$.  Thus, by definition, $\varphi(x,w)=(\varphi_w(x),w)$, and more generally $\varphi^n(x,w)=(\varphi_w^n(x),w)$ for $n\geq 0$.
\begin{notation}
\label{def:measure}
  Let $\GG$ be a group generated by $S\Subset \GG$. We let $\mu_p$ and $\widetilde{\mu}_p$ denote the Bernoulli measures, with parameter $p\in(0,1)$, on $\{\symb{0},\symb{1}\}^\GG$ and $\{\symb{0},\symb{1}\}^{ S\times \GG}$ respectively.
  In particular, identifying $A^\GG$ with $\{\symb{0},\symb{1}\}^\GG\times\{\symb{0},\symb{1}\}^{ S\times \GG}$ as before, $\mu_\half\times\widetilde{\mu}_\half$ stands for the uniform Bernoulli measure on $A^\GG$.
\end{notation}

The fact that $\varphi$ is not $(\mu_\half\times\widetilde{\mu}_\half)$-sensitive is easy to show (Proposition~\ref{prop:random-percolated additive:not-mu-sensitive}).  To prove that $\varphi$ is not $(\mu_\half\times\widetilde{\mu}_\half)$-equicontinuous, we use a percolation argument.
We will also use the following simple lemma.

\begin{lemma}
\label{lem:parity-of-iid-Bernoulli-is-Bernoulli}
    Let $n\geq 1$ and let $Z_1,Z_2,\ldots,Z_n$ be i.i.d.~Bernoulli random variables with parameter~$\half$.  Then, $Y\isdef (Z_1+Z_2+\cdots+Z_n) \bmod{2}$ is also a Bernoulli random variable with parameter~$\half$.
\end{lemma}

\subsection{Connection with percolation}
\label{ssection:connection}
In this subsection $\GG$ denotes a finitely generated group. We show how the propagation of information in the percolated additive CA associated to a set of generators~$S$ is linked to site percolation on $\cayley(\GG,S)$. To this end, we introduce two processes, one tracking the dependencies over time in the CA when the environment is random and the other exploring the cluster of the origin in site percolation.

Let $\varphi$ be the percolated additive CA on $\GG$ associated to $S$.
Since $\varphi$ is additive on its first coordinate (Observation~\ref{rem:percolated additive:additivity}), so is $\varphi^n$ for every $n\geq 0$.  Hence, for every $w\in\{\symb{0},\symb{1}\}^{\GG\times S}$ and $n\geq 0$, there exists a set $M_n(w)\Subset\GG$ such that
\begin{align}
\label{eq:CA-after-n-step:additive}
    \varphi_w^n(x)_e &= \Big(\sum_{g\in M_n(w)} x_g\Big) \bmod{2}
\end{align}
for every $x\in\{\symb{0},\symb{1}\}^\GG$. Notice that if for $g\in \GG$ we let $\delta_g \in \{\symb{0},\symb{1}\}^\GG$ be the configuration with value $\symb{1}$ on~$g$ and $\symb{0}$ everywhere else, then
\begin{align*}
    M_n(w) &=
        \big\{g\in\GG: \varphi_w^n(\delta_g)_e=\symb{1}\big\}
\end{align*}

\begin{observation}
\label{obs:dependence-process:recursive}
    We have
    \begin{align*}
        M_0(w) &= \{e\} \;, \\
        M_n(w) &= \big\{g\in\GG: \textup{$\abs[\big]{\{s\in S: \text{$gs^{-1}\in M_{n-1}(w)$ and $w_{gs^{-1}}(s)=\symb{1}$}\}}$ is odd}\big\} 
            \qquad \text{for $n\geq 1$.}
    \end{align*}
    In other words, $M_n(w)$ consists of all sites to which $w$ has an odd number of open bonds from $M_{n-1}(w)$.
\end{observation}

A probability measure $\nu$ on $\{\symb{0},\symb{1}\}^{\GG\times S}$ turns $(M_n)_{n\geq 0}$ into a stochastic process which we call the \Emph{dependence process} of $\varphi$. Let $M_{\leq n}(w)\isdef\bigcup_{i=0}^n M_i(w)$. We say that the process $(M_n)_{n\geq 0}$ \Emph{terminates} on $w$ if $M_{\leq n+1}(w)=M_{\leq n}(w)$ for some $n\geq 0$. Otherwise, we say that the process \Emph{survives}.

\begin{remark}
    The condition $M_{\leq n+1}(w)=M_{\leq n}(w)$ does not imply $M_{\leq n+k}=M_{\leq n}(w)$ for all $k\geq 0$.  However, the above strong notion of survival will be sufficient for our purpose.
\end{remark}

The following follows easily from Observation~\ref{obs:dependence-process:recursive} using induction.
\begin{observation}
\label{obs:dependence-process:measurability}
    For $n\geq 1$, the set $M_n(w)$ is uniquely determined by the restriction of $w$ to $M_{\leq n-1}(w)$.
\end{observation}

Next, we define a similar process based on site percolation.
Given $x\in\{\symb{0},\symb{1}\}^\GG$ and $n\geq 0$, we let $M'_n(x)$ denote the set of all sites $g\in\GG$ to which $x$ has an open path of length $n$ from $e$ to $g$.  We do not require $e$ to be open in such a path.
We let $M'_{\leq n}(x)\isdef\bigcup_{i=0}^n M'_i(x)$.

\begin{observation}
    We have
    \begin{align*}
        M'_0(x) &= \{e\} \;, \\
        M'_n(x) &= \big\{g\in\GG: \textup{$x_g=\symb{1}$ and $g=fs$ for some $f\in M'_{n-1}(x)$ and $s\in S$} \big\} 
            \qquad \text{for $n\geq 1$.}
    \end{align*}
    In other words, $M'_n(x)$ consists all open sites of~$x$ to which there is a bond from $M'_{n-1}(x)$.
\end{observation}

A probability measure $\mu$ on $\{\symb{0},\symb{1}\}^\GG$ turns $(M'_n)_{n\geq 0}$ into a stochastic process which we call the \Emph{(percolation) cluster exploration process}.  As before, we say that the process $(M'_n)_{n\geq 0}$ \Emph{terminates} or \Emph{survives} on $x$ depending on whether $M'_{n+1}(x)=M'_n(x)$ for some $n\geq 0$ or not.
Clearly, if $\mu$ percolates, then $(M'_n)_{n\geq 0}$ survives.

We are now ready to state the connection between the percolated additive CA and site percolation.


\begin{proposition}
\label{prop:dependence-process-vs-cluster-exploration-process}
    If the cluster exploration process $(M'_n)_{n\geq 0}$ (with measure $\mu_\half$) has positive probability of survival, then so does the dependence process $(M_n)_{n\geq 0}$ (with measure $\widetilde{\mu}_\half$).
\end{proposition}

\begin{proof}
    It suffices to construct a coupling of $\mu_\half$ and $\widetilde{\mu}_\half$ with the property that $M'_{\leq n}=M_{\leq n}$ almost surely for every $n\geq 0$. We shall do this by recursively sampling $\xv{w}$ from $\widetilde{\mu}_\half$ and $\xv{x}$ from $\mu_\half$ in such a way that $M'_{\leq n}(\xv{x})=M_{\leq n}(\xv{w})$ for all $n\geq 0$.
    As before, we think of the sites $g$ with $\xv{x}_g=\symb{1}$ as \Emph{open} sites and the bonds $(g,gs)$ with $\xv{w}_g(s)=\symb{1}$ as \Emph{open} bonds.
    The non-open sites and bonds are considered \Emph{closed}.

    We start by sampling the status of the site at the origin.  By definition, $M'_0=M_0=\{e\}$ irrespective of the values of $\xv{w}$ and $\xv{x}$.
    Let $A_0\isdef\{e\}$.
    Next, we sample the status of all the bonds exiting $e$.  We declare each site $s\in S\setminus\{e\}$ to be open if and only if the bond $(e,s)$ is open; otherwise, we declare $s$ to be closed.
    Clearly, $M_1\setminus\{e\}=M'_1\setminus\{e\}=\{s: \text{$(e,s)$ is open}\}$, hence $M_{\leq 1}=M'_{\leq 1}$.  Let $A_1\isdef A_0\cup S$.

    Let $n\geq 2$.  Suppose that by the end of the $(n-1)$st step, we have ensured that $M'_{\leq k}=M_{\leq k}$ for $k=0,1,\ldots,n-1$ and in the process we have sampled the status of all the sites and bonds in the subgraph induced by $A_{n-1}\isdef A_0\cup M_{\leq n-2}S$.
    At the $n$th step, we first sample the status of all the bonds from $M_{n-1}$ to $\GG\setminus A_{n-1}$ (equivalently, from $M_{\leq n-1}$ to $\GG\setminus A_{n-1}$).  We then declare a site $h\in M_{n-1}S\setminus A_{n-1}$ as open if the number of open bonds from $M_{n-1}$ to $h$ is odd; otherwise, $h$ is declared as closed.  By Lemma~\ref{lem:parity-of-iid-Bernoulli-is-Bernoulli}, each such site will be open with the correct probability of~$\half$.
    It follows directly from the definitions that $M'_n\setminus M'_{\leq n-1}=M_n\setminus M_{\leq n-1}$, hence $M'_{\leq n}=M_{\leq n}$.
    Furthermore, it is clear that by the end of this step, we have sampled the status of all the sites and bonds in $A_n\isdef A_0\cup M_{\leq n-1}S$.

    The sites and bonds whose status are not sampled at any step are irrelevant to the two processes $(M_n)_{n\geq 0}$ and $(M'_n)_{n\geq 0}$.  We can sample their status independently of one another.
    The claimed property of the constructed coupling now follows by induction.
\end{proof}

\subsection{Lowering the percolation threshold}

This section is devoted to proving the following result on the percolation threshold of finitely generated groups.

\begin{proposition}[Percolation with low threshold]
\label{prop:halfpercolation}
    Let $\GG$ be a finitely generated group with a non-trivial percolation threshold. For every $\alpha \in (0,1)$ there exists a set of generators $S\Subset \GG$ such that $p_{\critical}(\cayley(\GG,S))<\alpha$.
\end{proposition}
\noindent In other words, in every group where the percolation threshold is non-trivial, we have
\begin{align*}
    \inf_{S\Subset \GG} p_{\critical}\big(\cayley(\GG,S)\big) &= 0 \;.
\end{align*}
We remark that 
the corresponding statement regarding the opposite end
was recently proven false by Panagiotis and Severo~\cite{PS2023}.
Namely, they showed that there is a universal gap $\varepsilon_0>0$ such that every Cayley graph $\Gamma$ with non-trivial percolation threshold satisfies $p_{\critical}(\Gamma) \leq 1-\varepsilon_0$. 

Let us note that in the case where $\GG=\ZZ^d$, the conclusion of Proposition~\ref{prop:halfpercolation} can be proven easily by partitioning $\ZZ^d$ into hypercubic blocks of equal size and using the fact that these blocks form a lattice that is again isomorphic to~$\ZZ^d$.  Unfortunately, this argument does not seem to extend to arbitrary finitely generated groups.  We prove the proposition by considering the cases in which $\GG$ is nonamenable (Proposition~\ref{prop:infamous_nonamenable}) and amenable (Proposition~\ref{prop:infamous_amenable}) separately.

Let $\GG$ be a group and let $K\Subset \GG$ and $\delta >0$. A set $F\Subset \GG$ is said to be \Emph{$(K,\delta)$-invariant}, if
\begin{align*}
    \abs{FK\setminus F} &\leq \delta\abs{F} \;.
\end{align*}
Given $F,K\Subset\GG$, we let $\Interior_K(F)\isdef\{t\in F: tK\subseteq F\}$ and $\partial_K(F)\isdef F\setminus\Interior_K(F)$. 
A double counting argument shows that if $F$ is a $(K,\sfrac{\delta}{\abs{K}})$-invariant set, then $\abs[\big]{\partial_K(F)}< \delta \abs{F}$ (see e.g.,~\cite[Lemma 2.6]{DownarowiczHuczekZhang2019}).

A group $\GG$ is called \Emph{amenable} if for every pair $(K,\delta)$ there exists some $(K,\delta)$-invariant set $F\Subset \GG$; if this does not hold, we say that $\GG$ is \Emph{nonamenable}. An elementary computation shows that if $\GG$ is nonamenable, then for every $C >0$ one can find a set $K\Subset \GG$ such that $\abs{FK}> C\abs{F}$ for every $F \Subset \GG$ (see e.g.,~\cite[proof of Theorem 4.9.2]{ceccherini2010cellular}).

\subsubsection{The nonamenable case}

Let $\GG$ be a group and $S\Subset \GG$ a finite set of generators. Given $r \in \NN$, we say that $\Delta \subseteq \GG$ is
\begin{itemize}
    \item \Emph{$r$-separated} if for every distinct $g,h \in \Delta$ we have $\abs{g^{-1}h}_S\geq r$.
    \item \Emph{$r$-covering} if for every $h \in \GG$ there is a $g \in \Delta$ with $\abs{g^{-1}h}_S \leq r$.
\end{itemize}
Notice that a maximal $r$-separated set $\Delta\subseteq\GG$ is necessarily $r$-covering.

A \Emph{bipartite graph} is a graph $\Gamma$ whose vertex set
is a union of two disjoint sets $U$ and $V$ and all its edges
are between elements of $U$ and elements of $V$.
A perfect matching is a bijection $\varphi \colon U \to V$ with the property that $(u,\varphi(u))$ is an edge for every $u \in U$.  We denote the set of \Emph{neighbors} of a vertex $a\in U\sqcup V$ by $\neighbours(a)$.  More generally, given $A\subseteq U\sqcup V$, we let $\neighbours(A)\isdef\bigcup_{a\in A}\neighbours(a)$.
We say that $\Gamma$ is \Emph{locally finite} if for every $a\in U\sqcup V$, the set $\neighbours(a)$ is finite.



A necessary and sufficient condition for the existence of perfect matchings is provided by Hall's matching theorem~\cite{PHall1935,MHall1948}, which we now recall.
A proof of the following version of this theorem can be found in~\cite[Theorem H.3.6]{ceccherini2010cellular}.

\begin{knowntheorem}[Hall's matching theorem]\label{thm:Hall-matching}
    Let $\Gamma$ be a locally finite bipartite graph with vertex set $U\sqcup V$.
    Then, $\Gamma$ has a perfect matching if and only if the following conditions are satisfied:
    \begin{enumerate}[label={\textup{(\roman*)}}]
        \item \textup{(\Emph{left Hall condition})} $\abs{\neighbours(A)}\geq\abs{A}$ for every $A\Subset U$.
        \item \textup{(\Emph{right Hall condition})} $\abs{\neighbours(B)}\geq\abs{B}$ for every $B\Subset V$.
    \end{enumerate}
\end{knowntheorem}


Our proof of Proposition~\ref{prop:halfpercolation} in the nonamenable case relies on the following lemma.

\begin{lemma}[Geometric inflation]\label{lem:fat_nonamenable}
    Let $\GG$ be a nonamenable group, $S\Subset \GG$ a generating set, and $r\in\NN$. There exists $n\in\NN$ such that, for every $r$-covering subset $\Delta$ of~$\GG$, there exists a bijection $\varphi \colon \GG \to \Delta$ such that $\abs{g^{-1}\varphi(g)}_S\leq n$ for every~$g\in \GG$. 
\end{lemma}

\begin{proof}
    As usual, we let $B_n$ denote the centered ball of radius~$n$ in $\cayley(\GG,S)$.
    Since $\GG$ is nonamenable, we can find a finite set $K\Subset \GG$ such that for every $A\Subset \GG$, we have
    \begin{align*}
        \abs{AK} &\geq \abs{B_r}\cdot\abs{A} \;.
    \end{align*}
    Let $n\in \NN$ be large enough such that $K\subseteq B_n$, and consider the locally finite bipartite graph $\Gamma_n$, where the vertices are given by the disjoint union $\GG \sqcup \Delta$, and $(g,h)\in G \times \Delta$ is an edge if and only if $\abs{g^{-1}h}_S\leq n$. A map $\varphi$ which satisfies the requirements of the lemma is then given by a perfect matching in $\Gamma_n$.  Thus, to prove the lemma it suffices to verify the hypotheses of Hall's matching theorem (Theorem~\ref{thm:Hall-matching}).
    
    Clearly, for every $B\Subset\Delta$, the set $\neighbours(B)\subseteq\GG$ is at least as large as~$B$ because $(h,h)$ is an edge for every $h\in\Delta$.
    Hence, the right Hall condition is satisfied.

    To verify the left Hall condition, consider $A \Subset \GG$ and notice that
    \begin{align*}
        \neighbours(A) &= \Delta \cap \bigcup_{g \in A} B_n(g) = \Delta \cap A B_n \;.
    \end{align*}
    Since $\Delta$ is $r$-covering, we have
    \begin{align*}
        \abs{\neighbours(A)} &= \abs{\Delta \cap A B_n} \geq \frac{\abs{AB_n}}{\abs{B_r}} \;.
    \end{align*}
    As $K\subseteq B_n$, it follows that $\abs{AB_n}\geq \abs{AK} \geq \abs{A}\cdot\abs{B_r}$.
    Therefore, $\abs{\neighbours(A)} \geq \abs{A}$, which means the left Hall condition is also satisfied.
\end{proof}

Note that if in the lemma above $\Delta$ is also $(2\ell+1)$-separated for some $\ell \in \NN$, then the balls $B_\ell\big(\varphi(g)\big)$ (for $g \in \GG$) are disjoint and $\abs{\varphi(g)^{-1}\varphi(gs)}_S\leq 2n+1$ for every $s \in S$, thus the pair $(\Delta,\varphi)$ represents a ``geometric inflated copy of $\GG$''. We remark that variants of this lemma hold in uniformly discrete metric spaces with bounded geometry~\cite{Whyte1999}.  

\begin{proposition}[Nonamenable case]\label{prop:infamous_nonamenable}
    Let $\GG$ be a nonamenable group with a generating set $S\Subset \GG$, and let $0<\alpha<\beta<1$.
    There exists a generating set $S'\Subset \GG$ of $\GG$ such that
    \begin{align*}
        \mu_\beta\big(\{x\in\{\symb{0},\symb{1}\}^\GG: \textup{$x$ percolates in $\cayley(\GG,S)$}\}\big)
        &\leq
        \mu_\alpha\big(\{y\in\{\symb{0},\symb{1}\}^\GG: \textup{$y$ percolates in $\cayley(\GG,S')$}\}\big) \;.
    \end{align*}
\end{proposition}
\begin{proof}
    As before, we let $B_n$ denote the centered ball of radius $n$ in $\cayley(\GG,S)$.

    Let $\ell,r\in\NN$ be constants to be chosen later.
    Let $\Delta\subseteq\GG$ be $(2\ell+1)$-separated and $r$-covering. 
    By Lemma~\ref{lem:fat_nonamenable}, there exists a bijection $\varphi\colon \GG \to \Delta$ and an integer $m \geq 0$ such that $\varphi(g)\in gB_m$ for every $g\in\GG$.
    Consider the function $\zeta:\{\symb{0},\symb{1}\}^\GG\to\{\symb{0},\symb{1}\}^\GG$, where
    \begin{align*}
        \zeta(x)_g &\isdef
            \begin{cases}
                \symb{1}    & \text{$x_h=\symb{1}$ for some $h\in \varphi(g)B_\ell$,} \\
                \symb{0}    & \text{otherwise.}
            \end{cases}
    \end{align*}

    First, observe that 
    if $\zeta(x)$ percolates in $\cayley(\GG,S)$, then $x$ percolates in $\cayley(\GG,S')$, where $S'\isdef B_\ell^{-1}B_m^{-1}SB_m B_\ell=B_{2(\ell+m)+1}$.
    Indeed, suppose $(g_k)_{k=1}^\infty$ is a self-avoiding path in $\cayley(\GG,S)$ with $\zeta(x)_{g_k}=\symb{1}$ for all~$k$.
    Then, for each $k$, there exists an element $h_k\in\varphi(g_k)B_\ell$ such that $x_{h_k}=\symb{1}$.  Since $\varphi(g_k)\in g_k B_m$, we have $h_k\in g_kB_mB_\ell$.
    It follows that $h_{k+1}\in h_k(B_m B_\ell)^{-1}SB_mB_\ell = h_k S'$, thus $(h_k)_{k=1}^\infty$ is a path in $\cayley(\GG,S')$.  Since $\Delta$ is $(2\ell+1)$-separated, the sets $\varphi(g_k)B_\ell$ are disjoint, hence $(h_k)_{k=1}^\infty$ is self-avoiding.
    Therefore, $x$ percolates in $\cayley(\GG,S')$ as claimed.

    Next, observe that $\zeta\mu_\alpha=\mu_{\beta'}$ where $\beta'\isdef 1 - (1-\alpha)^{\abs{B_\ell}}$.  In other words, if $\xv{x}$ is distributed according to $\mu_\alpha$, then $\zeta(\xv{x})$ is distributed according to~$\mu_{\beta'}$.
    Indeed, since $\Delta$ is $(2\ell+1)$-separated, the sets $\varphi(g)B_\ell$ (for $g\in\GG$) are disjoint, which implies the values $\xv{x}_g$ (for $g\in\GG$) are independent.
    Furthermore, the probability that $\zeta(\xv{x})_g=\symb{1}$ is clearly $1-(1-\alpha)^{\abs{B_\ell}}=\beta'$.

    Combining the above two observations, we obtain that
    \begin{align*}
        \mu_{\beta'}\big(\{y\in\{\symb{0},\symb{1}\}^\GG: \textup{$y$ percolates in $\cayley(\GG,S)$}\}\big)
        &=
        \mu_\alpha\big(\{x\in\{\symb{0},\symb{1}\}^\GG: \textup{$\zeta(x)$ percolates in $\cayley(\GG,S)$}\}\big) \\
        &\leq
        \mu_\alpha\big(\{x\in\{\symb{0},\symb{1}\}^\GG: \textup{$x$ percolates in $\cayley(\GG,S')$}\}\big) \;.
    \end{align*}

    As nonamenable groups are infinite, we can now choose $\ell$ large enough such that $\beta'\geq\beta$ and let $r\geq 2\ell+1$ to guarantee the existence of~$\Delta$.
    The result then follows from the monotonicity of the percolation probability on the Bernoulli parameter.
\end{proof}

\begin{remark}[Alternative approach]
Benajmini and Schramm proved that
\begin{align*}
p_c(\cayley(\GG,S)) &\leq \frac{1}{1+h\big({\cayley(\GG,S)}\big)} \;,
\intertext{where} 
    h\big({\cayley(\GG,S)}\big) &= \inf_{F\Subset \GG}\frac{\abs{FS\setminus F}}{\abs{F}}
\end{align*}
is the \Emph{Cheeger constant} of the graph $\cayley(\GG,S)$~\cite[Theorem 2]{BS1996}.
Note that when $\GG$ is nonamenable, by choosing the set of generators~$S$ appropriately, the Cheeger constant can be made arbitrarily large.  This provides an alternative proof of Proposition~\ref{prop:infamous_nonamenable}. \end{remark}

\subsubsection{The amenable case}

Let $\GG$ be a group. A \Emph{tile set} is a finite collection $\mathcal{T} = \{T_1,\dots,T_n\}$ of finite subsets of $\GG$ which contain the identity.
A \Emph{tiling} of $\GG$ by $\mathcal{T}$ is a map $\tau \colon \GG \to  \mathcal{T} \cup \{\varnothing\}$ such that:
\begin{enumerate}[label={(\roman*)}]
    \item ($\tau$ is pairwise-disjoint) For every $g,h \in \GG$, if $g \neq h$ then $g\tau(g) \cap h\tau (h) =\varnothing$.
    \item ($\tau$ covers $\GG$) For every $g \in \GG$, there exists an $h \in \GG$ such that $g \in h\tau(h)$.
\end{enumerate}

We shall use the following result of Downarowicz, Huczek and Zhang~\cite[Theorem 4.3]{DownarowiczHuczekZhang2019}.

\begin{knowntheorem}[Amenable tilings]\label{teorema_tiling_exacto}
	Let $\GG$ be a countable amenable group. For every $F\Subset\GG$ and $\delta >0$,
    there exists a tiling of $\GG$ by a tile set~$\mathcal{T}$ whose elements are all $(F,\delta)$-invariant.
\end{knowntheorem}

A \Emph{coupling} of two probability measures $p$ and $q$ on measurable spaces $U$ and $V$ refers to a probability measure $r$ on $U\times V$ that has marginals $p$ and $q$, or equivalently, to a pair of random variables with joint distribution~$r$.  We shall also use the following elementary version of Strassen's theorem on the existence of couplings~\cite[Theorem~11]{Strassen1965}, which is equivalent to Hall's matching theorem.
\begin{knowntheorem}[Strassen's coupling]
\label{thm:strassen:finite}
    Let $p$ and $q$ be probability measures on finite sets $U$ and $V$ respectively, and let $\prec$ be a binary relation on $U\times V$.
    There exists a coupling $r$ of $p$ and $q$ satisfying $r\big(\{(a,b): a\prec b\}\big)=1$ if and only if for every $A\subseteq U$, we have $p(A)\leq q\big(\{b\in V: \text{$a\prec b$ for some $a\in A$}\}\big)$.
\end{knowntheorem}

\begin{proposition}[Amenable case]\label{prop:infamous_amenable}
    Let $\GG$ be an infinite amenable group with a generating set~$S\Subset \GG$, and let $0<\alpha<\beta<1$.
    There exists a generating set $S'\Subset \GG$ of $\GG$ such that
    \begin{align*}
        \mu_\beta\big(\{x\in\{\symb{0},\symb{1}\}^\GG: \textup{$x$ percolates in $\cayley(\GG,S)$}\}\big)
        &\leq
        \mu_\alpha\big(\{y\in\{\symb{0},\symb{1}\}^\GG: \textup{$y$ percolates in $\cayley(\GG,S')$}\}\big) \;.
    \end{align*}
\end{proposition}
\begin{proof}
    The proof is via a coupling argument.
    
    Let $\delta >0$ be a constant to be determined later.
    By Theorem~\ref{teorema_tiling_exacto}, there exists a finite tile set $\mathcal{T}= \{T_1,\dots,T_k\}$ where every $T_i$ is $(S,\delta)$-invariant and $\GG$ admits a tiling $\tau \colon \GG \to \mathcal{T} \cup \{\varnothing\}$.
    Let $\Delta\isdef\tau^{-1}(\mathcal{T})$ be the set of centers of the tiles in~$\tau$.  This set is naturally endowed with a graph structure in which there is an edge from $g\in\Delta$ to $h\in\Delta$ if $g\tau(g)S\cap h\tau(h)\neq\varnothing$.
    Let $E$ denote the edges of this graph.
    Set $S'\isdef\bigcup_{(g,h)\in E} \big(g\tau(g)\big)^{-1}h\tau(h)$ so that for every $(g,h)\in E$ and $g'\in g\tau(g)$, we have $g'S\supseteq h\tau(h)$.  Note that $S'$ is finite because up to translations, there are only finitely many local configurations of neighboring tiles in~$\tau$.

    Let us now define two functions $\zeta,\eta:\{\symb{0},\symb{1}\}^\GG\to\{\symb{0},\symb{1}\}^\Delta$, where
    \begin{align*}
        \zeta(x)_g &\isdef
            \begin{cases}
                \symb{1}    & \text{if $x_h=\symb{1}$ for some $h\in \partial_S\big(g\tau(g)\big)$,} \\
                \symb{0}    & \text{otherwise.}
            \end{cases} \\
        \eta(y)_g &\isdef
            \begin{cases}
                \symb{1}    & \text{if $y_h=\symb{1}$ for some $h\in g\tau(g)$,} \\
                \symb{0}    & \text{otherwise,}
            \end{cases}
    \end{align*}
    Observe that for $x,y\in\{\symb{0},\symb{1}\}^\GG$:
    \begin{enumerate}[label={(\alph*)}]
        \item If $x\in\{\symb{0},\symb{1}\}^\GG$ percolates in $\cayley(\GG,S)$, then $\zeta(x)$ percolates in the graph $(\Delta,E)$.
        \item If $\eta(y)$ percolates in $(\Delta,E)$, then $y$ percolates in $\cayley(\GG,S')$.
    \end{enumerate}
    Let $\Omega = \{\symb{0},\symb{1}\}^\GG\times \{\symb{0},\symb{1}\}^\GG$. We show that, for a suitable choice of $\delta$, there exists a coupling $\nu$ of $\mu_\beta$ and $\mu_\alpha$ such that
    \begin{align*}
        \nu\big(\{(x,y) \in \Omega: \zeta(x)\leq \eta(y) \}\big) &= 1 \;.
    \end{align*}
    (The inequality $\zeta(x)\leq \eta(y)$ means $\zeta(x)_g\leq \eta(y)_g$ for every $g\in\Delta$.)
    If so, then
    \begin{align*}
        \MoveEqLeft
        \mu_\beta\big(\{ x\in\{\symb{0},\symb{1}\}^\GG: \textup{$x$ percolates in $\cayley(\GG,S)$ } \}\big) \\
        &=
        \nu\big(\{(x,y)\in \Omega: \textup{$x$ percolates in $\cayley(\GG,S)$ and $\zeta(x)\leq \eta(y)$}\}\big) \\
        &\leq
        \nu\big(\{(x,y)\in \Omega: \textup{$\zeta(x)$ percolates in $(\Delta,E)$ and $\zeta(x)\leq \eta(y)$}\}\big) \\
        &\leq
        \nu\big(\{(x,y)\in \Omega: \textup{$\eta(y)$ percolates in $(\Delta,E)$ and $\zeta(x)\leq \eta(y)$}\}\big) \\
        &\leq
        \nu\big(\{(x,y)\in \Omega: \textup{$y$ percolates in $\cayley(\GG,S')$ and $\zeta(x)\leq \eta(y)$}\}\big) \\
        &=
        \mu_\alpha\big(\{y\in\{\symb{0},\symb{1}\}^\GG: \textup{$y$ percolates in $\cayley(\GG,S')$}\}\big) \;,
    \end{align*}
    which would prove the proposition.

    To this end, choose $\delta<\log(1-\alpha)/\log(1-\beta)$ and the tile set $\mathcal{T}$ accordingly so as to ensure that $(1-\beta)^{\abs{\partial_S(T_i)}}\geq(1-\alpha)^{\abs{T_i}}$ for each tile $T_i\in\mathcal{T}$.
    To construct the coupling $\nu$, we couple the marginals of $\mu_\beta$ and $\mu_\alpha$ on each tile of $\tau$ independently.
    For each tile $T_i\in\mathcal{T}$, let $U_i=V_i\isdef\{\symb{0},\symb{1}\}^{T_i}$.
    Define a binary relation $\prec$ on $U_i\times V_i$ by letting $a\prec b$ if and only if either $a_{\partial_S T_i}=\symb{0}^{\partial_S(T_i)}$ or $b_{T_i}\neq\symb{0}^{T_i}$.
    Let $p_i$ and $q_i$ be the Bernoulli measures with parameters $\beta$ and $\alpha$ on $U_i$ and $V_i$ respectively.  From the choice of $\delta$, it follows that
    $p_i(A)\leq q_i\big(\{b\in V_i: \text{$a\prec b$ for some $a\in A$}\}\big)$
    for every $A\subseteq U_i$.  Hence, by Strassen's coupling theorem (Theorem~\ref{thm:strassen:finite}), there exists a coupling $r_i$ of $p_i$ and $q_i$ such that $r_i\big(\{(a,b): a_i\prec b_i\}\big)=1$.  Note that for $x,y\in\{\symb{0},\symb{1}\}^\GG$, we have $\zeta(x)\leq \eta(x)$ if and only if $(g^{-1}x)_{\tau(g)}\prec (g^{-1}y)_{\tau(g)}$ for each $g\in\Delta$, hence the coupling $\nu$ thus constructed has the desired property.
\end{proof}

\subsection{Proof of the characterization}
\label{ssectionproof}

Let us first verify that the percolated additive CAs are not $(\mu_\half\times\widetilde{\mu}_\half)$-sensitive (see Notation~\ref{def:measure}).
This is in fact true in more generality.
\begin{proposition}[Not $\mu$-sensitive]
\label{prop:random-percolated additive:not-mu-sensitive}
       Let $\GG$ be a group generated by $S\Subset \GG$. The percolated additive CA on $\GG$ associated to $S$
       is not sensitive with respect to any fully supported measure.
\end{proposition}
\begin{proof}
    Let $\varphi\colon A^\GG\to A^\GG$ denote the percolated additive CA on $\GG$ associated to $S$ 
    and $\mu$ be a full-support measure on $A^\GG$.
    Let $(x,w)\in A^\GG$ be a configuration in which $x_g=\symb{0}$ and $w_g(s)=\symb{0}$ for every $g\in\GG$ and $s\in S$.  Clearly, $(x,w)$ is a fixed point of~$\varphi$.  Furthermore, it is easy to see that $[(x,w)_F]=C\big((x,w),F,\varphi\big)$ for every $F\Subset\GG$.  Since $\mu$ is fully supported, it follows that $\mu\big(C\big((x,w),F,\varphi\big)\big)=\mu\big([(x,w)_F]\big)>0$ for every $F\Subset\GG$, which means $\varphi$ is not $\mu$-sensitive (see Remark~\ref{rem:musens:sure}).
\end{proof}

\begin{proposition}[Not $\mu$-equicontinuous]
\label{prop:not-mu-equicontinuous-if-dependence-process-survives}
    Let $\GG$ be a group generated by $S\Subset \GG$, $\nu$ a probability measure on $\{\symb{0},\symb{1}\}^{\GG\times S}$ and $\varphi$ the percolated additive CA on $\GG$ associated to $S$. 
    If the dependence process of $\varphi$ with measure $\nu$ has a positive probability of survival, then $\varphi$ is not $(\mu_\half\times\nu)$-equicontinuous.
\end{proposition}
\begin{proof}
    Let $\mu\isdef\mu_\half\times\nu$.  As before, we let $B_n$ denote the centered ball of radius $n$ in $\cayley(\GG,S)$.
    
    Let $Q$ denote the set of all environment configurations $w\in\{\symb{0},\symb{1}\}^{\GG\times S}$ on which the dependence process of $\varphi$ survives.  By assumption, $\nu(Q)>0$ hence $\mu(\{\symb{0},\symb{1}\}^\GG\times Q)>0$.
    We claim that
    \begin{align}
    \label{eq:not-mu-equicontinuous:bounded-away}
        \mu\Big(C\big((x,w),\{e\},\varphi\big)\given[\Big][(x,w)_{B_n}]\Big) &\leq \half
    \end{align}
    for every $w\in Q$, $x\in\{\symb{0},\symb{1}\}^\GG$ and $n\geq 0$, which means that $(x,w)$ is not a point of density of~$C\big((x,w),\{e\},\varphi\big)$ with respect to the co-final chain $(B_n)_{n=1}^\infty$.
    This would thus imply that $\varphi$ is not $\mu$-equicontinuous by Proposition~\ref{prop:mu-equicontinuity}.

    So, let $w\in Q$, $x\in\{\symb{0},\symb{1}\}^\GG$ and $n\geq 0$.
    Since the dependence process of $\varphi$ survives on~$w$, there exists a time $t$ such that $M_t(w)\setminus B_n$ is non-empty.  Take the smallest such~$t$.
    From 
    Observation~\ref{obs:dependence-process:measurability}
    and the choice of $t$ it follows that $M_t(w)$ is uniquely determined by the restriction of $w$ to~$B_n$.  In other words, $M_t(w')=M_t(w)$ for every $w'\in[w_{B_n}]$.
    Hence, according to~\eqref{eq:CA-after-n-step:additive}, for every $(x',w')\in [(x,w)_{B_n}]$, we have
    \begin{align*}
        \varphi_{w'}^t(x')_e &= \Big(\sum_{g\in M_t(w)} x'_g\Big) \bmod{2}
        =
            \Big(\sum_{g\in M_t(w)\cap B_n} x_g + \sum_{g\in M_t(w)\setminus B_n} x'_g\Big) \bmod{2}
    \end{align*}
    Since $M_t(w)\setminus B_n$ is non-empty, by Lemma~\ref{lem:parity-of-iid-Bernoulli-is-Bernoulli}, we have
    \begin{align*}
        \MoveEqLeft
        \mu_\half\bigg(\Big\{x'\in\{\symb{0},\symb{1}\}^\GG: \Big(\sum_{g\in M_t(w)\setminus B_n} x'_g\Big) \bmod{2}=\symb{0}\Big\}\bigg)
        \\
        &=
        \mu_\half\bigg(\Big\{x'\in\{\symb{0},\symb{1}\}^\GG: \Big(\sum_{g\in M_t(w)\setminus B_n} x'_g\Big) \bmod{2}=\symb{1}\Big\}\bigg)
        = \half \;.
    \end{align*}
    Therefore,
    \begin{align*}
        \MoveEqLeft
        \mu\Big(\big\{(x',w')\in A^\GG:\varphi_{w'}^t(x')_e=\symb{1}\big\}\given[\Big][(x,w)_{B_n}]\Big) \\
        &=
        \mu\Big(\big\{(x',w')\in A^\GG:\varphi_{w'}^t(x')_e=\symb{0}\big\}\given[\Big][(x,w)_{B_n}]\Big)
        = \half \;,
    \end{align*}
    from which~\eqref{eq:not-mu-equicontinuous:bounded-away} follows.
\end{proof}

We can now prove the first main result of this paper:
\emph{an infinite, finitely generated group satisfies Gilman's dichotomy if and only if it is virtually $\ZZ$.}

\begin{proof}[Proof of Theorem~\ref{thm:main-result}]
    If $\GG$ is virtually~$\ZZ$, then by Theorem~\ref{thm:virtually_Z_satisfies_dichotomy}, it satisfies Gilman's dichotomy for every $\GG$-ergodic probability measure.
    Now, let $\GG$ be an infinite, finitely generated group that is not virtually~$\ZZ$. By Theorem~\ref{thm:percolation}, every Cayley graph of $\GG$ has a non-trivial percolation threshold. By Proposition~\ref{prop:halfpercolation} there exists a set of generators $S\Subset \GG$ so that $p_{\critical}(\cayley(\GG,S))<1/2$.
    
    Let $\varphi$ be the percolated additive CA on $\GG$ associated to $S$. 
     By Proposition~\ref{prop:random-percolated additive:not-mu-sensitive}, we have that $\varphi$ is not $(\mu_\half\times\widetilde{\mu}_\half)$-sensitive.
    
    As $p_{\critical}(\cayley(\GG,S))<1/2$, Proposition~\ref{prop:dependence-process-vs-cluster-exploration-process} implies that the dependence process $(M_n)_{n \geq 0}$ has a positive probability of survival with respect to $\widetilde{\mu}_\half$, thus from Proposition~\ref{prop:not-mu-equicontinuous-if-dependence-process-survives} we conclude that $\varphi$ is not $(\mu_\half\times\widetilde{\mu}_\half)$-equicontinuous, and hence the result.    
\end{proof}


\section{Gilman's dichotomy for countable groups}
\label{sec:gilman-on-countable-groups}


In this section, we extend the characterization of the groups that satisfy Gilman's dichotomy to cover all countable groups.

Let $\varphi \colon A^{\GG}\to A^{\GG}$ be a CA. Theorem~\ref{thm:curtis} ensures that there exists a set $K\Subset \GG$ and a local function $f\colon A^K\to A$ such that $\varphi(x)_g = f\big((g^{-1}x)_K\big)$ for every $g \in \GG$. If we let $\HH \isdef \generatedby{K}$ be the subgroup of~$\GG$ generated by $K$, then $f$ also induces a CA $\widetilde{\varphi}\colon A^{\HH}\to A^{\HH}$ through $\widetilde{\varphi}(x)_h = f\big((h^{-1}x)_K\big)$ for every $h \in \HH$. Conversely, every CA on $A^{\HH}$ can be extended to a CA on $A^{\GG}$ using the same local map $f$.

\begin{lemma}[Sensitivity and equicontinuity and subgroups]
\label{lema:extension_from_subgroups}
    Let $\HH$ be a subgroup of $\GG$.
    Consider a set $K\Subset \HH$ and a local rule $f\colon A^K \to A$.  Let $\varphi$ and $\widetilde{\varphi}$ be the CA induced by $f$ on $\GG$ and $\HH$ respectively, and let $\mu$ and $\widetilde{\mu}$ be Bernoulli measures with the same coordinate-wise marginal distribution on $A^{\GG}$ and $A^{\HH}$ respectively.
    Then,
    \begin{enumerate}[label={\textup{(\roman*)}}]
        \item $\varphi$ is $\mu$-sensitive if and only if $\widetilde{\varphi}$ is $\widetilde{\mu}$-sensitive.
        \item $\varphi$ is $\mu$-equicontinuous if and only if $\widetilde{\varphi}$ is $\widetilde{\mu}$-equicontinuous.
    \end{enumerate}
\end{lemma}
\begin{proof}
    Let $W$ be a set of representatives from the left cosets of $\HH$ in $\GG$, that is, a set such that every $g\in\GG$ has a unique representation as $g=wh$ for some $w\in W$ and $h\in\HH$.
    \begin{enumerate}[label={\textup{(\roman*)}}]
        \item First, suppose that $\widetilde{\varphi}$ is not $\widetilde{\mu}$-sensitive.
        Let $F\Subset\GG$.
        Partitioning $F$ according to the $\HH$-cosets it intersects, we can write $F=\bigcup_{i=1}^n w_iF_i$ for some $n\in\NN$, distinct $w_1,w_2,\ldots,w_n\in W$ and $F_1,F_2,\ldots,F_n\Subset\HH$.
        Since $\widetilde{\varphi}$ is not $\widetilde{\mu}$-sensitive, for each $i$ there exists a configuration $\widetilde{x}^{(i)}\in A^\HH$ such that $\widetilde{\mu}\big(C(\widetilde{x}^{(i)},F_i,\widetilde{\varphi})\big)>0$.
        Choose a configuration $x\in A^\GG$ such that $(w_i^{-1} x)_\HH=\widetilde{x}^{(i)}$ for each~$i$.
        Note that
        \begin{align*}
            C(x,F,\varphi) &=
                \bigcap_{i=1}^n C(x,w_i F_i,\varphi) \;.
        \end{align*}
        Since $K\Subset\HH$, the values of $\varphi^n(x)$ on $w_i F_i$ depend only on the values of $x$ on $w_i\HH$.
        Hence, the sets $C(x,w_i F_i,\varphi)$ are independent with respect to~$\mu$, that is,
        \begin{align*}
            \mu\big(C(x,F,\varphi)\big) &=
            \prod_{i=1}^n
                \mu\big(C(x,w_i F_i,\varphi)\big) \;.
        \end{align*}
        Moreover, because of the $\GG$-invariance of~$\mu$, we have
        \begin{align*}
            \mu\big(C(x,w_i F_i,\varphi)\big) &=
            \mu\big(w_i^{-1} C(x,w_i F_i, \varphi)\big) =
            \mu\big(C(w_i^{-1}x,F_i,\varphi)\big) \;.
        \end{align*}
        Lastly, since $F_i\Subset\HH$, we have
        \begin{align*}
            \mu\big(C(w_i^{-1}x,F_i,\varphi)\big) &=
             \widetilde{\mu}\big(C((w_i^{-1}x)_\HH,F_i,\widetilde{\varphi})\big) =
             \widetilde{\mu}\big(C(\widetilde{x}^{(i)},F_i,\widetilde{\varphi})\big)
             > 0 \;.
        \end{align*}
        Putting all together, we obtain that
        $\mu\big(C(x,F,\varphi)\big)>0$.
        Therefore, $\varphi$ is not $\mu$-sensitive
        (see Remark~\ref{rem:musens:sure}).

        Conversely, suppose that $\widetilde{\varphi}$ is $\widetilde{\mu}$-sensitive.  Then, there exists a set $F\Subset\HH$ such that $\widetilde{\mu}\big(C(\widetilde{x},F,\widetilde{\varphi}\big)=0$ for every $\widetilde{x}\in A^\HH$.
        Thus, for every $x\in A^\GG$, we have
        \begin{align*}
            \mu\big(C(x,F,\varphi)\big) &=
            \widetilde{\mu}\big(C(x_\HH,F,\widetilde{\varphi})\big)
            = 0 \;,
        \end{align*}
        which means $\varphi$ is $\mu$-sensitive.


        \item Suppose that $\varphi$ is $\mu$-equicontinuous.  Let $F\Subset\HH$.
        Define
        $X \isdef \big\{x\in A^\GG: \mu\big(C(x,F,\varphi)\big)>0\big\}$.
        Since $\varphi$ is $\mu$-equicontinuous, we have $\mu(X)=1$.
        Since $K,F\subseteq\HH$, we have
        \begin{align*}
            \mu\big(C(x,F,\varphi)\big) &=
                \widetilde{\mu}\big(C(x_\HH,F,\widetilde{\varphi})\big) 
        \end{align*}
        and
        $X=\widetilde{X}\times A^{\GG\setminus\HH}$ for some measurable $\widetilde{X}\subseteq A^\HH$.
        Clearly, $\widetilde{\mu}(\widetilde{X})=1$.
        Moreover, $\widetilde{\mu}\big(C(\widetilde{x},F,\widetilde{\varphi})\big)>0$ for all $\widetilde{x}\in\widetilde{X}$.
        Therefore, $\widetilde{\varphi}$ is $\widetilde{\mu}$-equicontinuous (see Remark~\ref{rem:mueq:quantifiers}).
        
        Conversely, suppose that $\widetilde{\varphi}$ is $\widetilde{\mu}$-equicontinuous.
        Let $F\Subset\GG$.
        As in the previous part, we can write $F=\bigcup_{i=1}^n w_iF_i$ for some $n\in\NN$, distinct $w_1,w_2,\ldots,w_n\in W$ and $F_1,F_2,\ldots,F_n\Subset\HH$,
        so that
        \begin{align*}
            \mu\big(C(x,F,\varphi)\big) &=
            \prod_{i=1}^n
                \widetilde{\mu}\big(C((w_i^{-1}x)_\HH, F_i,\widetilde{\varphi})\big)
        \end{align*}
        for every $x\in A^\GG$.
        Since $\widetilde{\varphi}$ is $\widetilde{\mu}$-equicontinuous, for each $i$, there exists a measurable set $X_i\subseteq A^\HH$ with $\widetilde{\mu}(X_i)=1$
        such that $\widetilde{\mu}\big(C(\widetilde{x},F_i,\widetilde{\varphi})\big)>0$ for every $\widetilde{x}\in X_i$.
        Let $X\isdef\big\{x\in A^\GG: \text{$(w_i^{-1}x)_\HH\in X_i$ for $i=1,2,\ldots,n$}\big\}$.
        Clearly, $\mu(X)=1$.
        Furthermore, $\mu\big(C(x,F,\varphi)\big)>0$ for each~$x\in X$.
        We conclude that $\varphi$ is $\mu$-equicontinuous.
        \qedhere
    \end{enumerate}
\end{proof}


We are now ready to prove the general characterization: \emph{a countable group satisfies Gilman's dichotomy if and only if it is locally virtually cyclic.}

\begin{proof}[Proof of Theorem~\ref{thm:main2}]
    Let $\GG$ be a countable group.
    
    First, suppose that $\GG$ has a finitely generated subgroup $\HH$ that is not virtually cyclic.
    Then, Theorem~\ref{thm:main-result} provides a CA on $A^{\HH}$ that is neither sensitive nor equicontinuous with respect to the uniform Bernoulli measure on $A^{\HH}$. By Lemma~\ref{lema:extension_from_subgroups}, the extension of that CA to $\GG$ is neither sensitive nor equicontinuous with respect to the uniform Bernoulli measure on $A^{\GG}$.
    
    Conversely, suppose that every finitely generated subgroup of $\GG$ is virtually cyclic.
    Then, every CA on $\GG$ induces a CA on a finitely generated subgroup $\HH$, which is either a virtually $\ZZ$ or a finite group.  By Theorem~\ref{thm:virtually_Z_satisfies_dichotomy}, we know that the dichotomy holds if $\HH$ is virtually $\ZZ$.  If $\HH$ is finite, it is clear that every CA is equicontinuous and not sensitive with respect to every probability measure on $A^{\HH}$, thus the dichotomy holds trivially.  We conclude again using Lemma~\ref{lema:extension_from_subgroups} that the dichotomy holds for~$\GG$ as well.
\end{proof}

Interesting examples of non-finitely generated locally virtually cyclic groups (and hence where Gilman's dichotomy holds) include the additive group of rational numbers $\QQ$, the $p$-adic rationals $\ZZ \left[\sfrac{1}{p}\right]$, the Pr\"{u}fer $p$-groups $\ZZ\left[\sfrac{1}{p}\right]/\ZZ$, and the group $S_{\infty}$ of finitely supported permutations of a countably infinite set.

\section{Further remarks and questions}
\label{sec:remarks}

\subsection{Examples of $\mu$-sensitive CA on groups}


On every countable group, one can find CA that are equicontinuous with respect to the uniform Bernoulli measure (e.g., the identity).  However, not every countable group admits a CA that is sensitive with respect to the uniform Bernoulli measure.  For instance, a CA on a finite group cannot be sensitive with respect to any measure, and by Lemma~\ref{lema:extension_from_subgroups}, the same is true for locally finite groups, at least with respect to Bernoulli measures.
Below we provide examples of $\mu$-sensitive CA for a class of groups.

\begin{proposition}
    Let $\GG$ be a group that has a non-torsion element. Then, there exists a CA on~$\{\symb{0},\symb{1}\}^{\GG}$ that is sensitive with respect to the uniform Bernoulli measure $\mu_{\half}$.
\end{proposition}
\begin{proof}
    Let $h\in \GG$ be a non-torsion element of~$\GG$ and let $\tau\colon \{\symb{0},\symb{1}\}^{\GG}\to \{\symb{0},\symb{1}\}^{\GG}$ be given by
    $\tau(x)_g \isdef x(gh)$ for every $g\in\GG$.
    Observe that, for $x,y\in \{\symb{0},\symb{1}\}^{\GG}$, we have $y \in C(x,\{e\},\tau)$ if and only if $x_T = y_T$, where $T \isdef \{h^n : n \geq 0\}$.  As $h$ is not a torsion element, the set $T$ is infinite.
    It follows that $\mu_\half\big(C(x,\{e\},\tau)\big)=0$ for every $x\in \{\symb{0},\symb{1}\}^{\GG}$, which means $\tau$ is $\mu_{\half}$-sensitive.
\end{proof}


\begin{question}
    Do infinite, finitely generated torsion groups admit CA that are
    sensitive with respect to the uniform Bernoulli measure?
\end{question}

\subsection{Site-percolated additive CA and odd percolation}

A perhaps more natural candidate for a CA that does not satisfy Gilman's dichotomy is the following:

\begin{example}[Site-percolated additive CA]
\label{exp:random-site-XOR-CA}
    Let $S$ be a finite generating set for a group~$\GG$.
    Let $A\isdef\{\symb{0},\symb{1},\symb{\star}\}$ and consider the CA given by the map $\varphi'\colon A^\GG\to A^\GG$, where
    \begin{align*}
        \varphi'(x)_g &\isdef
            \begin{cases}
                \sum_{s\in S} \pi(x_{gs}) \bmod{2}      & \text{if $x_g\in\{\symb{0},\symb{1}\}$,} \\
                \symb{\star}                        & \text{if $x_g=\symb{\star}$,}
            \end{cases}
            \qquad\text{for every $g \in \GG$,}
    \end{align*}
    where $\pi(\symb{\star})=\pi(\symb{0})\isdef\symb{0}$ and $\pi(\symb{1})\isdef\symb{1}$.  We view the sites with $\star$ as being \Emph{closed} and the other sites  as \Emph{open}.  Thus, the closed sites remain closed, and the open sites are updated to the sum modulo~$2$ of their open neighbors. 
\end{example}

An argument similar to that of Proposition~\ref{prop:random-percolated additive:not-mu-sensitive} shows that the site-percolated additive CA is not sensitive with respect to any fully supported measure.
For the special case of the group~$\ZZ^2$,
a result of Bramson and Neuhauser on random perturbations of CA~\cite{BN1994} can be used to show that, for a specific choice of the set of generators, the site-percolated additive CA is not equicontinuous with respect to some Bernoulli measures, and thus violates Gilman's dichotomy.
We conjecture that the same is true for all groups with non-trivial percolation threshold.
This would arguably be simpler than the CA in Example~\ref{exp:random-percolated additive-CA}.

\begin{question}
    Is there a non-virtually cyclic, finitely generated group $\GG$ on which, for every choice of the generating set $S$, the site-percolated additive CA is equicontinuous with respect to every Bernoulli measure?
\end{question}



To contrast it with the site-percolated additive CA, we may refer to the CA of Example~\ref{exp:random-percolated additive-CA} as the \Emph{bond-}percolated additive CA.
The site- and bond-percolated additive CA are closely related to special percolation models which we call \Emph{odd (bond or site) percolation}.


Let $\Gamma$ be a locally finite graph and $w\in\{\symb{0},\symb{1}\}^{E(\Gamma)}$ a configuration of \Emph{open} and \Emph{closed} edges.
We say that $w$ \Emph{odd-percolates} from a vertex $a\in V(\Gamma)$ to a vertex $b\in V(\Gamma)$ if there exists an $\ell\in\NN$ such that the number of open paths of length $\ell$ from $a$ to $b$ is odd.
If the latter holds for some $a\in V(\Gamma)$ and infinitely many choices of~$b\in V(\Gamma)$, we say that $w$ \Emph{odd-percolates}.
Lastly, we say that a probability measure~$\mu$ on~$\{\symb{0},\symb{1}\}^{E(\Gamma)}$ \Emph{odd-percolates} if $\mu$ assigns a positive probability to the set of configurations that odd-percolate.
Odd percolation for configurations of open and odd vertices and for measures on such configurations are defined analogously.

Now, consider a group~$\GG$ with generating set~$S\Subset\GG$.
From the discussion of Subsection~\ref{ssection:connection}, it is evident that a configuration $w$ odd-percolates in $\cayley(\GG,S)$ if and only if
the dependence process of the bond-percolated additive CA survives on~$w$.
In particular, Proposition~\ref{prop:not-mu-equicontinuous-if-dependence-process-survives} can be rephrased as follows: if a probability measure~$\nu$ on the environment configurations odd-percolates, then the bond-percolated additive CA is not equicontinuous with respect to~$\mu_\half\times\nu$.
A similar correspondence holds between site odd percolation and the site-percolated additive CA.
This leads to the following question.


\begin{question}
    Let $\GG$ be a finitely generated group that is not virtually $\ZZ$. Is there a generating set $S\Subset \GG$ and a non-trivial Bernoulli measure on $\{\symb{0},\symb{1}\}^\GG$ that odd-percolates in $\cayley(\GG,S)$?
\end{question}
\noindent According to Proposition~\ref{prop:dependence-process-vs-cluster-exploration-process}, the answer to the corresponding question for bond odd percolation is positive.

\subsection{Topological dichotomy}

As mentioned in the introduction, K\r{u}rka proved a topological analogue of Gilman's dichotomy~\cite{Kurka1997}.

\begin{theorem}[K\r{u}rka's dichotomy]
    Every CA $\varphi\colon A^{\ZZ}\to A^{\ZZ}$ is either sensitive or almost equicontinuous. 
\end{theorem}

An adaptation of K\r{u}rka's proof along with Lemma~\ref{lema:tecnicotrivial} can be used to extend this result to all virtually $\ZZ$ groups.  A complete proof was given by Audouard~\cite{audouard}.
On the other hand, Sablik and Theyssier constructed an example on~$\ZZ^2$ that is not sensitive and has no equicontinuity points~\cite{ST2010}.
This construction relies on the geometry of~$\ZZ^2$ in a non-trivial fashion, making it difficult to generalize to other groups beyond~$\ZZ^d$.
Nonetheless, it has been conjectured that every countable group that is not virtually cyclic admits a CA that violates K\r{u}rka's dichotomy~\cite[Conjecture 5.2.25]{Bitar-thesis_2024}.

\begin{question}\label{question:kurka}
Let $\GG$ be a finitely generated group which is not virtually cyclic. Is there a CA on $\GG$ that is neither sensitive nor almost equicontinuous? 
\end{question}


Despite the analogy with our Theorem~\ref{thm:main-result}, the percolated additive CA cannot be used to answer this question as they are always almost equicontinuous.

\begin{proposition}\label{prop:percolated_is_almost_equicontinuous}
    Let $\GG$ be a group generated by $S\Subset \GG$. The associated percolated additive CA is almost equicontinuous. 
\end{proposition}

\begin{proof}
    Given $z=(x,w)\in A^{\GG}$ and $F\Subset\GG$, consider the configuration $z'=(x,w')\in [z_F]$ in which
    \begin{align*}
        w'_g(s) &\isdef
            \begin{cases}
                w_g(s)  &\text{if $g\in F$,} \\
                \symb{0}       &\text{otherwise,}
            \end{cases}
    \end{align*}
    for every $g\in\GG$ and $s\in S$.
    Note that $z'$ is an equicontinuity point of the CA because for every $E\Subset\GG$ with $FS\subseteq E$, we have $[z'_E]\subseteq C(z',E,\varphi)$.
    It follows that the set of equicontinuity points of the CA is dense.
    However, if the set of equicontinuity points of a topological dynamical system is dense, it must also be residual~\cite[Proposition 2.30]{Kurka2003a}.
    We conclude that the CA is almost equicontinuous.
\end{proof}

As a consequence of Gilman's and K\r{u}rka's dichotomy theorems, every almost equicontinuous CA on~$\ZZ$ is equicontinuous with respect to every fully supported probability measure (see also~\cite[Propositions 3.4 and 3.5]{Gilman1987}). However,
this fails for groups that are not virtually cyclic as the percolated additive CA is almost equicontinuous but not $\mu_\half\times\widetilde{\mu}_\half$-equicontinuous.

\begin{question}
    Is there an almost equicontinuous CA that is not equicontinuous with respect to any non-trivial Bernoulli measure?
\end{question}

\subsection{Equicontinuity points with respect to a measure}

There is a natural notion of an equicontinuity point with respect to a measure.  Let $\varphi\colon A^{\GG}\to A^{\GG}$ be a CA and $\mu$ a probability measure on $A^{\GG}$. We say that $x\in A^{\GG}$ is a \Emph{$\mu$-equicontinuity point} of $\varphi$ if $\mu\big(C(x,F,\varphi)\big)>0$ for every $F\Subset\GG$. 
With this definition, $\varphi$ is $\mu$-equicontinuous if and only if $\mu$-almost all its configurations are $\mu$-equicontinuity points.

For a $\GG$-invariant measure~$\mu$, the set of $\mu$-equicontinuity points of~$\varphi$ is $\GG$-invariant.
Thus, when $\mu$ is $\GG$-ergodic, this set has either full or null measure.
The percolated additive CA has equicontinuity points with respect to every fully supported $\GG$-invariant measure (use Proposition~\ref{prop:percolated_is_almost_equicontinuous}), even though the set of such points has null measure.
This leads to the following question in analogy with the construction of Sablik and Theyssier~\cite{ST2010}.

\begin{question}
    Let $\GG$ be a non locally virtually cyclic group.
    Does there exist a CA on $\GG$ that has no $\mu$-equicontinuity points and is not $\mu$-sensitive for some Bernoulli measure~$\mu$?
\end{question} 

Proposition~\ref{prop:mu-equicontinuity} provides a characterization of $\mu$-equicontinuity in terms of $\mu$-density points.  We do not know if the local version of that proposition holds.
\begin{question}
 Let $(J_n)_{n\in\NN}$ be a co-final chain. Is it true that $x\in A^{\GG}$ is a $\mu$-equicontinuity point if and only $x$ is a point of $\mu$-density of $C(x,F,\varphi)$ with respect to $(J_n)_{n\in \NN}$ for every $F\Subset\GG$ ?  
\end{question}
 
We note that in other references~\cite{Gilman1987,GarciaRamos2017}, the local definition of $\mu$-equicontinuity for $\GG=\ZZ$ is based on $\mu$-density points with respect to the co-final chain of intervals centered in the origin.

\subsection{Reversible cellular automata}

A CA is called \Emph{reversible} if it has an inverse that is itself a CA.  Since the space of configurations is compact and Hausdorff, every bijective CA is reversible.
The percolated additive CA (Example~\ref{exp:random-percolated additive-CA}) is neither injective nor surjective, and thus is far from being reversible.
It is natural to ask:

\begin{question}
    Let $\GG$ be a group that is not virtually cyclic.
    Does there exist a reversible CA that is neither sensitive nor equicontinuous with respect to some Bernoulli measure?
\end{question}

We propose a candidate, which is a variant of the percolated additive CA.
Let $S\Subset \GG$ be a set of generators.
Let $A \isdef \{\symb{0},\symb{1}\}\times \{\symb{0},\symb{1}\}\times \{\symb{0},\symb{1}\}^S$ and define a CA $\varphi \colon A^{\GG}\to A^{\GG}$ by
\begin{align*}
    \varphi(x,y,w)_g &\isdef
        \bigg(y_g, \Big(x_g+\sum_{s\in S} w_{g}(s)\cdot y_{gs}\Big) \bmod{2}, w_g\bigg),
        \qquad\text{for every $g \in \GG$.}
\end{align*}
A straightforward verification shows that this CA is reversible.
Furthermore, one can argue as in Proposition~\ref{prop:random-percolated additive:not-mu-sensitive} that this CA is not sensitive with respect to any fully supported measure.
We suspect that an adaptation of our argument can be used to show that when $\GG$ is not virtually~$\ZZ$, this CA (with an appropriate choice of $S$) is not equicontinuous with respect to some Bernoulli measure.

\bibliographystyle{plainurl}
\bibliography{bibliography}

\bigskip
\Addresses

\end{document}